\documentclass[10pt,oneside]{article}
\usepackage{times}
\usepackage{rotating}

\usepackage{cancel}

\usepackage{booktabs,amsthm, amsmath, amssymb, amsfonts, graphicx, epsfig, enumitem}

\usepackage{mathrsfs}

\usepackage{enumitem}

\usepackage{color}
\definecolor{labelkey}{rgb}{0,0,1}

\usepackage[colorlinks=true, pdfstartview=FitV, linkcolor=blue,
            citecolor=blue, urlcolor=blue]{hyperref}

\usepackage[textsize=scriptsize]{todonotes}

\usepackage[numbers]{natbib}

\usepackage{url}
\makeatletter\def\url@leostyle{%
 \@ifundefined{selectfont}{\def\UrlFont{\sf}}{\def\UrlFont{\scriptsize\ttfamily}}} \makeatother

\usepackage{accents, comment}

\usepackage[normalem]{ulem}

\newtheorem{theorem}{Theorem}
\newtheorem{proposition}[theorem]{Proposition}
\newtheorem{lemma}[theorem]{Lemma}

\theoremstyle{definition}
\newtheorem{definition}[theorem]{Definition}
\newtheorem{example}[theorem]{Example}
\theoremstyle{remark}
\newtheorem{remark}[theorem]{Remark}

\numberwithin{equation}{section}
\numberwithin{theorem}{section}

\def\cA{\mathcal{A}}
\def\cB{\mathcal{B}}
\def\cC{\mathcal{C}}

\def\cE{\mathcal{E}}
\def\cF{\mathcal{F}}

\def\cK{\mathcal{K}}

\def\cP{\mathcal{P}}
\def\cQ{\mathcal{Q}}
\def\cR{\mathcal{R}}
\def\cS{\mathcal{S}}

\def\cX{\mathcal{X}}

\def\bN{\mathbb{N}}

\def\bR{\mathbb{R}}

\newcommand{\set}[1]{\{#1\}} 
\newcommand{\Set}[1]{\left\{#1\right\}} 
\renewcommand{\mid}{\;|\;}              
\newcommand{\Mid}{\;\Big | \;}          
\newcommand{\norm}[1]{\left\Vert#1\right\Vert}  
\newcommand{\abs}[1]{\left\vert#1\right\vert}   

\DeclareMathOperator*{\esssup}{ess\,sup} 
\DeclareMathOperator*{\essinf}{ess\,inf} 
\DeclareMathOperator*{\esslimsup}{ess\,\limsup}

\DeclareMathOperator*{\hypo}{hypo\,}

\def\geqc{\succcurlyeq}

\title{Dynamic Assessment Indices}

\author{Tomasz R. Bielecki \\
\small{Department of Applied Mathematics,}\\[-0.3ex]
\small{Illinois Institute of Technology,}\\[-0.3ex]
\small{Chicago, 60616 IL, USA}\\[-0.3ex]
\url{bielecki@iit.edu}\\
\and
Igor Cialenco\\[-0.3ex]
\small{Department of Applied Mathematics,}\\[-0.3ex]
\small{Illinois Institute of Technology,}\\[-0.3ex]
\small{Chicago, 60616 IL, USA}\\[-0.3ex]
\url{igor@math.iit.edu} \\
\and
Samuel Drapeau\\
\small{Humboldt-Universit\"at Berlin,}\\[-0.3ex]
\small{Unter den Linden 6,}\\[-0.3ex]
\small{10099 Berlin, Germany}\\[-0.3ex]
\url{drapeau@math.hu-berlin.de} \\
\and
Martin Karliczek\\
\small{Humboldt-Universit\"at Berlin,}\\[-0.3ex]
\small{Unter den Linden 6,}\\[-0.3ex]
\small{10099 Berlin, Germany}\\[-0.3ex]
\url{karliczm@math.hu-berlin.de}}

\date{ August 15, 2014 \\
{\footnotesize First Circulated: June 15, 2013}}

\begin{document}
\maketitle

\begin{abstract}
\noindent This paper provides a unified framework, which allows, in particular,  to study the structure of dynamic monetary risk measures and dynamic acceptability indices.
The main mathematical tool, which we  use here, and which allows us to significantly generalize existing results is the theory of $L^0$-modules.
In the first part of the paper we develop the general theory and provide a robust representation of conditional assessment indices, and in the second part we apply this theory to dynamic acceptability indices acting on stochastic processes.

{\footnotesize
\bigskip

}
\end{abstract}

\tableofcontents

\begin{quote}
\textit{``Each definition is a piece of secret ripped from Nature by the human spirit. I insist on this: any complicated thing being illumined by definitions, being laid out in them, being broken up in pieces, will be separated into pieces completely transparent even to a child, excluding foggy and dark parts that our intuition whispers to us while acting, separating the object into logical pieces, then only can we move further towards new success due to definitions.''}

\hfill Nikolai N. Luzin \footnotesize{(quote from Loren R. Graham's ``Naming Infinity'')}
\end{quote}

\section*{Introduction}
\addcontentsline{toc}{section}{Introduction}
\markboth{\uppercase{Introduction}}{\uppercase{Introduction}}

This paper provides a study of Assessment Indices (AIs) in a discrete time dynamic framework. Assessment indices are meant to evaluate the trade-off between reward opportunities and danger of losses.\footnote{In fact,  assessment indices can be used to assess (to measure) various risks in the classical sense, for example just as monetary risk measures, such as V@R, do, but they can also be used to assess the trade-off between monetary risks and corresponding rewards, just as acceptability indices, such as Gain-to-Loss Ratio, do. Thus, the universe of AIs encompasses both the classical risk measures and the classical acceptability indices. However, here we adopt the universal interpretation of risk in the spirit of \citet{DrapeauKupper2010}, thus we look at risks as appreciation of danger of losses vis-a-vis the potential rewards, which renders our understanding of the assessment indices.} Consequently, the two basic operational paradigms, that underlie the mathematical theory of assessment indices, are well appreciated truths that in any kind of economic/financial activities:
\vspace{-.2cm}
\begin{itemize}\addtolength{\itemsep}{-0.7\baselineskip}
\item[(A)] Diversification is better than concentration;
\item[(B)] Greater success is better than lesser success.
\end{itemize}
\vspace{-.2cm}
\noindent These two stylized key paradigms translate mathematically into quasiconcavity and monotonicity properties of an AI. In the static case, these two paradigms were studied in the context of preferences in \citet{Cerreia-Vioglioa2009,Cerreia-VioglioMaccheroniMarinacciMontrucchio2008a,DrapeauKupper2010}. The numerical representations corresponding to preference orderings satisfying properties (A) and (B) cover, among others, risk measures (cf. \citet{ArtznerDelbaenEberHeath1999,FollmerSchied2002b,FrittelliGianin2002}), as well as acceptability indices (cf. \citet{ChernyMadan2009}).

Since the main motivation for our study of assessment indices comes from the area of analysis of risks and rewards propagating in time then, as stated already, we are engaging in this paper in study of AIs in the dynamic set-up.  Thus, just as in the case of \emph{dynamic} (classical) risk measures and \emph{dynamic} acceptability indices, we call them \emph{Dynamic Assessment Indices} (DAIs).  In contrast to the static case, the study of DAIs bears additional conceptual difficulties related to the conditionality and to the need for adequate intertemporal assessment of risk and rewards propagating in time.

In the present work, we significantly extend previous studies regarding assessment indices to the conditional/dynamic setting, which, in particular, allows us to apply our theory to study DAIs acting on discrete time stochastic processes. The main mathematical tool, which we  use here in order to derive extension results of \citet{DrapeauKupper2010} to the conditional setting, is the theory of $L^0$--modules that was originated in \citet{FilipovicKupperVogelpoth2009} and in \citet{KupperVogelpoth2009}. Similar extension problem has been also studied in \citet{FrittelliMaggis2010,FrittelliMaggis2011}, \citet{BionNadal2009,BiaginiBion-Nadal2012}. Here, we provide a study in the general setting of locally convex topological $L^0$-modules inspired by the methods and techniques of \citep{DrapeauKupper2010}.

In many ways, the present paper continues and builds upon research of other people that has been presented in numerous works. For obvious reasons we can't provide here the comprehensive list of all these works. Besides the papers that we have already mentioned above, we think that the following works should be brought to the reader's attention:
\citet{CheriditoKupper2009}, \citet{AcciaioFollmerPenner2010}, \citet{FrittelliGianin2004}, \citet{Bion2008}, \citet{BionNadal2009,BiaginiBion-Nadal2012},
\citet{MadanCherny2010}, \citet{BCIR2012}, \citet{Detlefsen2005},
\citet{CheriditoDelbaenKupper2005,CheriditoDelbaenKupper2004,CheriditoDelbaenKupper2006a},
\citet{FrittelliMaggis2010,FrittelliMaggis2011}, \citet{Cerreia-Vioglioa2009,Cerreia-VioglioMaccheroniMarinacciMontrucchio2008a},
\citet{FrittelliScandolo2006}, \citet{PenotVolle1990}.

The paper is organized as follows. In Section~\ref{sec01} we introduce a set of some underlying concepts that will be used throughout the paper. Section~\ref{sec:AIGEneral} provides the main contribution of our work in the context of general theory of conditional assessment indices defined on locally convex topological $L^0$-modules and taking values in $\bar L^0$. In particular, Theorem~\ref{thm:robrep} furnishes robust representation, characterization indeed, for an upper semicontinuous conditional assessment index.
This is a novel and important result, which generalizes corresponding result obtained in the static (not conditional) setting in \cite{DrapeauKupper2010}.
The road leading to Theorem~\ref{thm:robrep}, which at first sight seems to be similar to what was done in \cite{DrapeauKupper2010}, has not been an easy one, as is seen from all the technical results (cf. Appendix, in particular) that we needed to obtain this robust dual representation. In particular, we provide a full duality result for conditionally increasing functions and their general left and right inverses.
The main hurdle lies in the central issue of locality, that is delicate and had to be handled with outmost care.
The results regarding scale invariant indices and the results regarding certainty equivalents, presented in Section~\ref{sec:scaleInv} and in Section~\ref{sec024}, respectively, are new, interesting, and useful.
In Section~\ref{sec03} we apply our general theory to study DAIs for discrete time stochastic processes. This comes in two flavors. First, in Section~\ref{sec031}, we apply the results of Section~\ref{sec:AIGEneral} almost verbatim, considering dynamic assessment index mapping processes into sequence of processes, and by making a very natural choice of $L^0$ space to be the space of stopped processes. Analysis of the robust representation result derived in this section brings about an interesting insight regarding the nature of the locality property: indeed, requiring locality relative to $\mathscr{O}^t$ (cf. Section~\ref{sec031}) implies that  $\alpha^t_t$ assess only the future (relative to $t$) of the process and, for $s<t$, $\alpha^t_s$ is just a function of the value at time $s$ of the assessed process. This is a drawback as for some applications this may be an unwanted feature. To overcome this drawback we adapt in Section~\ref{sec:PathDep} the theory of Section~\ref{sec:AIGEneral}, to the case of so called path dependent DAI, which maps processes into processes. They may, in particular, help a decision maker (eg. investor, or regulator), who is willing to design a DAI that at each time explicitly accounts for the past evolution of the underlying process that is being assessed (see Example~\ref{exep:pathdep}).
In Section \ref{sec04} we study strongly time consistent path dependent assessment indices, that satisfy some additional properties. The corresponding certainty equivalent is used to derive a relevant version of the dynamic programming principle, which characterizes the strong time consistency in this case.
Section \ref{sec:Examples} provides illustrating examples that we consider both interesting and important. We examine here a version of dynamic gain-to-loss ratio, which is a scale invariant DAI, and, in particular, we provide a robust representation for it. Two additional examples are given as well. Finally, in the Appendix, we present a variety of mathematical results, which underlie our theory. Appendix also contains proofs of some auxiliary technical results stated in the main body of the paper.

\section{Preliminaries}\label{sec01}

Let $( \Omega,\mathscr{G},P )$ be a probability space.
By $\mathscr{G}_{+}$ we denote the set of all events $A \in \mathscr{G}$ with $P[A]>0$.
If not otherwise specified, the notation $[A_i]\subseteq \mathscr{G}$ stands for a countable partition $(A_i)_{i\in \mathbb{N}}\subseteq \mathscr{G}$ of $\Omega$.
By $ L^0$, and $\bar{L}^0$ we denote the spaces of all $\mathscr{G}$-measurable random variables with values in $( -\infty,\infty)$ and $[-\infty,\infty]$, respectively.
As usual,  we identify random variables, which are equal  $P$-almost surely.
The relations $m>n$ and $m\geq n$ for  two random variables $m,n\in\bar{L}^0$ are to be  understood in the $P$-almost sure sense, that is, $P[m>n]=1$ and $P[m\geq n]=1$ respectively.
We also define   $L^0_{+} :=\set{m \in L^0 \mid m\geq 0}$ and $L^0_{++}:=\set{m \in L^0 \mid m>0}$.

As already stated in the Introduction, we are working in the setting of $L^0$-modules.
We refer to \citep{FilipovicKupperVogelpoth2009,KupperVogelpoth2009}, where this theory was initiated, for further details.
The space $L^0$ is a lattice ordered ring on which we, throughout the paper, consider the topology induced by the  balls
\begin{equation*}
	B_{\varepsilon}\left(  m \right):=\Set{n \in L^0 : \abs{m-n}\leq \varepsilon},\quad m \in L^0, \text{and }\varepsilon \in L^0_{++},
\end{equation*}
making $L^0$ to be a topological ring\footnote{That is, both the addition and scalar multiplication are continuous mappings with respect to the product topology.}.

From this point on, $\mathcal{X}$ denotes an $L^0$-module.
Given a set $\mathcal{C}\subseteq \mathcal{X}$, its \emph{$\sigma$-stable hull} is defined as
\begin{equation}
	\sigma\left( \mathcal{C} \right):=\Set{\sum 1_{A_i}X_i : \left[ A_i \right]\subseteq \mathscr{G}\text{ and }(X_i) \subseteq \mathcal{C}}.
	\label{}
\end{equation}
It  holds $\mathcal{C}\subseteq \sigma(\mathcal{C})$.
A set $\mathcal{C}\subseteq \mathcal{X}$ is called \emph{$\sigma$-stable} if $\mathcal{C}=\sigma(\mathcal{C})$.
A set $\mathcal{C}\subseteq \mathcal{X}$ is called \emph{$L^0$-convex} if $\lambda X+\left( 1-\lambda \right)Y\in\mathcal{C}$ for any $\lambda \in L^0$ with $0\leq \lambda \leq 1$ and $X,Y \in \mathcal{C}$.
By definition, $\mathcal{C}$ is $\sigma$-stable if and only if $\sum 1_{A_i}X_i \in \mathcal{C}$ for every $[A_i]\subseteq \mathscr{G}$ and $(X_i)\subseteq \mathcal{C}$.
In the following, $\mathcal{K}\subseteq \mathcal{X}$ will be an $L^0$-convex cone\footnote{That is, $\lambda X \in \mathcal{K}$ for any $\lambda \in L^0_{++}$ and $X \in \mathcal{K}$.} containing $0$.
Such an $L^0$-convex cone defines an $L^0$-module preorder\footnote{That is, $\lambda X+Z\succcurlyeq \lambda Y+Z$ for any $\lambda\in L^0_+$ and $Z \in \mathcal{X}$, whenever $X\succcurlyeq Y$ for $X,Y \in \mathcal{X}$.} $\succcurlyeq$ on $\mathcal{X}$, given by $X\succcurlyeq Y$ if $X-Y \in \mathcal{K}$.
We say a set $\mathcal{C}\subseteq \mathcal{X}$ is \emph{monotone} with respect to $\mathcal{K}$, or just monotone if there is no ambiguity about $\mathcal{K}$, if $\mathcal{C} +\mathcal{K} =\mathcal{C}.$

Working with (quasi)concave functions, we adopt the convention, $\infty -\infty:=-\infty$ and $0\cdot\pm\infty=0$.
We say that a function $F:\mathcal{X}\to \bar{L}^0$ is
\begin{itemize}
	\item \emph{$L^0$-local} if $F\left( 1_A X +1_{A^c}Y\right)=1_A F\left( X \right)+1_{A^c}F(Y)$;
	\item \emph{$L^0$-quasiconcave} if $F(\lambda X+\left( 1-\lambda \right)Y)\geq F(X)\wedge F(Y)$;
	\item \emph{$L^0$-concave} if $F(\lambda X+(1-\lambda)Y)\geq \lambda F(X)+(1-\lambda)F(Y)$;
	\item \emph{monotone with respect to $\mathcal{K}$} if $F(X) \geq F(Y)$, whenever $X\succcurlyeq Y$;
\end{itemize}
for any $X,Y \in \cX$, $\lambda \in L^0$ and $0\leq \lambda \leq 1$, and any $A \in \mathscr{G}$.
It can be shown that $F$ is local if and only if
\begin{equation}
	F\left( \sum 1_{A_i}X_i \right)=\sum 1_{A_i}F\left( 1_{A_i}X_i \right)=\sum 1_{A_i}F\left( X_i \right),
	\label{eq:localitywichtig}
\end{equation}
for every $[A_i]\subseteq \mathscr{G}$ and $(X_i)\subseteq \mathcal{X}$, as well as if and only if
\begin{equation}
	1_{A}F\left(X \right)=1_{A}F\left( 1_{A}X \right),
	\label{eq:localitywichtig2}
\end{equation}
for every $A \in \mathscr{G}$ and $X \in \mathcal{X}$.
A local function  $F$  of two arguments is called  \emph{jointly local}.

We further say that $F$ is
\begin{itemize}
	\item	\emph{$L^0$-linear} if $F$ takes values in $L^0$ and $F\left( m X+n Y \right)=m F\left( X \right)+n F\left( Y \right)$;
	\item \emph{positive homogeneous} if $F(\lambda X) = \lambda F(X)$;
	\item \emph{scale invariant} if $F(\lambda X)= F(X)$;
	\item \emph{cash additive in direction of $\kappa \in \mathcal{K}\setminus 0$} if $F(X + m\kappa) = F(X) + m$;
\end{itemize}
for any $X,Y \in \mathcal{X}$, any $m,n \in L^0$, and any $\lambda \in L^0_{++}$.

We now suppose that $\mathcal{X}$ is a locally $L^0$-convex topological $L^0$-module, see \citep[Definition 2.2]{FilipovicKupperVogelpoth2009}.
We denote by $\mathcal{X}^\ast$ its $L^0$-dual, that is, the set of all continuous $L^0$-linear functionals from $\mathcal{X}$ to $L^0$.
The $L^0$-dual  $\mathcal{X}^\ast$ is an $L^0$-module itself.
The weak topology, denoted by $L^0$-$\sigma\left( \mathcal{X},\mathcal{X}^\ast \right)$,  is the coarsest topology in $\cX$ for which the mappings
\begin{equation*}
	X \mapsto Z\left( X \right),\quad X \in \mathcal{X},
\end{equation*}
are continuous for any $Z\in\mathcal{X}^\ast$.

For a function $F:\mathcal{X}\to \bar{L}^0$ and for $m\in\bar{L}^0$,  we denote by $\cA^m$ the corresponding upper level set, that is  $\mathcal{A}^m:=\set{X \in \mathcal{X} \mid F(X)\geq m}$.
A function $F:\mathcal{X}\to \bar{L}^0$ is \emph{upper semicontinuous} if its upper level sets $\mathcal{A}^m$ are closed for all $m \in \bar{L}^0$.

It was shown in \cite{FilipovicKupperVogelpoth2009,KupperVogelpoth2009} that $F:\mathcal{X}\to \bar{L}^0$ is $L^0$-quasiconcave or monotone if and only if its upper level sets $\mathcal{A}^m$ are $L^0$-convex or monotone, for any $m \in \bar{L}^0$.
It is also known that $F$ is $L^0$-concave (resp. $L^0$-local) if and only if its hypograph $\hypo( F):= \set{(X,m)\in \mathcal{X}\times \bar{L}^0 \mid\alpha(X)\geq m}$ is $L^0$-convex\footnote{Even if $\bar{L}^0$ is not an $L^0$-module, using the convention  $\infty-\infty=\infty$ and $0 \cdot \infty=0$ on $\bar{L}^0$ we get the analogous results.} (resp.  $\sigma$-stable).

A set $\mathcal{B}\subseteq L^0$ is \emph{upward directed}, respectively \emph{downward directed}, if $X\wedge Y$, respectively $X\vee Y$, belongs to $\mathcal{B}$, for any $X,Y \in \mathcal{B}$.
In case of an upward directed, respectively downward directed, set, its essential supremum, respectively essential infimum, is attained by an increasing, respectively decreasing, sequence in this set, see \citep[Appendix A5]{FollmerSchiedBook2004}.
Similar results hold true for family of sets. If $(A_i)\subseteq \mathcal{G}$ is upward, respectively downward, directed with respect to the inclusion preorder, then there exists  essential supremum\footnote{That is, if $B \in \mathcal{G}$ is such that $A_i \subseteq B\subseteq A$ for all $i$, it holds $P[A\Delta B]=0$.}, respectively essential infimum, $A \in \mathcal{G}$, see \citep[Lemma 2.9]{FilipovicKupperVogelpoth2009}.

Throughout this paper, if no confusion may arise, we will often drop the reference to $L^0$ for all concepts from convex analysis.

\section{Robust Representation of Conditional Assessment Indices}\label{sec:AIGEneral}
In this section we follow the lines of \citep{DrapeauKupper2010}, extending the setup and the results presented therein to the \textit{conditional} case.
In the rest of this section we fix a cone $\cK\subseteq\cX$, and the monotonicity will be understood with respect to this cone.

\subsection{Conditional Assessment Indices and Conditional Risk Acceptance Family}
The main object studied in this paper is the conditional assessment index defined as follows.
\begin{definition}	
	A \emph{conditional assessment index} is a function $\alpha:\mathcal{X}\to \bar{L}^0$, which is local, quasiconcave, and monotone.\footnote{Recall that all concepts, such as  quasiconcave, local, etc., are understood in the $L^0$-sense.}
\end{definition}
Analogously  to the one-to-one relation between risk measures and risk acceptance families discussed in \citep{DrapeauKupper2010}, we also obtain a one-to-one relation, stated in Theorem \ref{th:AccIandARF}, between conditional assessment indices and conditional risk acceptance families defined below.
\begin{definition} A \textit{conditional risk acceptance family} is a family $\mathcal{A}:=(\mathcal{A}^m)_{m\in \bar{L}^0}$ of sets in $\cX$, which is
	\begin{itemize}
		\item \emph{convex:} $\mathcal{A}^{m}$ is convex, for any $m \in \bar{L}^0$;
		\item \emph{decreasing:} $\mathcal{A}^{m} \subseteq \mathcal{A}^{n}$, for any $n,m \in \bar{L}^0$ such that $m\geq n$;
		\item \emph{monotone:} $\mathcal{A}^{m}+\mathcal{K}=\mathcal{A}^{m}$, for any $m \in \bar{L}^0$;
		\item \emph{jointly $\sigma$-stable:} $\mathcal{A}=(\mathcal{A}^m)_{m\in \bar{L}^0}=\set{(X,m) \in \mathcal{X}\times \bar{L}^0: X \in \mathcal{A}^m}\subseteq \mathcal{X}\times \bar{L}^0$ is $\sigma$-stable;
		\item \emph{left-continuous:} for every $m \in \bar{L}^0$, the following identity  holds true
$$
				\mathcal{A}^m=1_{B(m)}\bigcap_{\substack{n<m\text{ on }B(m) \\ n=-\infty \text{ on } B^c(m)}  }\mathcal{A}^{n}+1_{B^c(m)}\mathcal{X},
$$
where $B(m)=\set{m>-\infty}$.

	\end{itemize}
\end{definition}

\begin{remark}\label{rem:ss}
	Note that the joint $\sigma$-stability of $\cA$ is equivalent to the property that
	\begin{equation*}
		\sum 1_{A^i}\cA^{m_i} = \mathcal{A}^{\sum 1_{A^i} m^i },
	\end{equation*}
	for any sequence $(m^i)\subseteq \bar{L}^0$, and any  $[A^i]\subseteq \mathscr{G}$.
	In particular, taking $m^i=m, \ i\in\bN,$ we get that $\sum 1_{A^i}\cA^{m} = \mathcal{A}^{m}$, and consequently we obtain that the set $\cA^m$ is $\sigma$-stable.
\end{remark}

 Theorem \ref{th:AccIandARF} below will play a central role in the proof of the robust representation theorem (cf. Section~\ref{sec022}).
\begin{theorem} \label{th:AccIandARF}
	Given  a conditional  assessment index $\alpha$, the family $\mathcal{A}_{\alpha}=\left( \mathcal{A}_{\alpha}^{m}\right)_{m \in \bar{L}^0}$ of sets defined by
	\begin{equation}
		\mathcal{A}_{\alpha}^{m} :=\Set{X \in \mathcal{X} \, : \, \alpha\left( X \right)\geq m},\quad m \in \bar{L}^0,
		\label{eq:AccIndandARF01}
	\end{equation}
	is a conditional risk acceptance family.

Conversely, given a conditional risk acceptance family $\mathcal{A}=\left( \mathcal{A}^{m}\right)_{m \in \bar{L}^0}$, the function $\alpha_{\mathcal{A}}:\mathcal{X}\to \bar{L}^0$ defined by
	\begin{equation}
		\alpha_{\mathcal{A}}\left( X \right):=\esssup\Set{m \in \bar{L}^0 \, : \,  X  \in \mathcal{A}^{m}  },\quad X \in \mathcal{X},
		\label{eq:AccIndandARF02}
	\end{equation} 	
	is a conditional assessment index.

Furthermore, with the previous notation,  $\alpha_{\mathcal{A}_{\alpha}}=\alpha$ and $\mathcal{A}_{\alpha_{\mathcal{A}}}=\mathcal{A}$.
\end{theorem}

\begin{remark}\label{rem:1to1fritelli}
	In the above result, $\mathcal{A}^m$ has to be indexed by $m\in\bar{L}^0$ rather than $m\in L^0$ to get a one-to-one correspondence.
	Indeed, suppose that our probability measure $P$ can be extended to $\mathscr{G}_1\supseteq\mathscr{G}$.
Let $\cX:=L^\infty_{\mathscr{G}}(\mathscr{G}_1)$, where $L^\infty_{\mathscr{G}}(\mathscr{G}_1)$ is defined as $L^\infty_{\cF}(\cE)$ in \cite[Section~4.2]{KupperVogelpoth2009}.
We  take an $A\in\mathscr{G}$, with $0<P[A]<1$.
The agent assess its utility exponentially conditioned on event $A$ whereas with a logarithm on $A^c$, that is
\begin{equation*}
    \alpha(X)=E\left[ 1_A\left( 1-e^{-X}) \right)+1_{A^c}\ln\left( X \right) \right], X\in L^\infty_{\mathscr{G}}(\mathscr{G}_1).
\end{equation*}
It follows that $\alpha(0)=1_A-\infty 1_{A^c}$ showing that $0 \not \in \mathcal{A}^m$ for every $m \in L^0$.
If we only take $m\in L^0$ into account, it would follow that $\alpha_{\mathcal{A}_\alpha}(0)=\infty\neq \alpha(0)$ and so the duality would fail.
\end{remark}

\begin{remark}
	Note that a version of  Theorem~\ref{th:AccIandARF} has been derived in \citep{FrittelliMaggis2011}.
	However, for a risk acceptance family we require joint $\sigma$-stability and an indexing by $\bar{L}^0$ rather than $L^0$; see also Remark~\ref{rem:1to1fritelli}.

	In contrast to the approach in the proof of the robust representation in \citep{FrittelliMaggis2011}, here the starting point for the Robust Representation Theorem~\ref{thm:robrep} will be the one-to-one correspondence between conditional assessment indices and conditional risk acceptance families stated in Theorem~\ref{th:AccIandARF}.
	We also note that Theorem~\ref{th:AccIandARF} is the conditional version of \citep[Theorem 1.7]{DrapeauKupper2010}. 	
\end{remark}

\begin{proof}
	\begin{enumerate}[label=\textit{Step \arabic*:},fullwidth]
		\item Let $\alpha$ be a conditional assessment index, and let the family of acceptance sets $\mathcal{A}_{\alpha}^{m}$, where $m \in \bar{L}^0,$ be defined as in \eqref{eq:AccIndandARF01}.
			By definition, $\mathcal{A}_{\alpha}$ is  decreasing.
			Furthermore, for any $m\in\bar{L}^0$, the set $\mathcal{A}^m_\alpha$ is convex and monotone, since it is an upper level set of a quasiconcave and monotone function.

			Next we show that $\mathcal{A}_{\alpha}$ is jointly $\sigma$-stable.
			Let $[A_i]\subseteq \mathscr{G}$ and $(X_i,m_i)_{i\in\bN}\subseteq \mathcal{A}_{\alpha}$, in particular $\alpha(X_i)\geq m_i, \ i\in\bN$.
			By definition of $\mathcal{A}_{\alpha}$, by locality of $\alpha$, and by \eqref{eq:localitywichtig}, it follows that $(X,m):=\sum 1_{A_i}(X_i,m_i)=(\sum 1_{A_i}X_i,\sum 1_{A_i}m_i)$ fulfills
			\begin{equation*}
				\alpha(X)=\alpha\left( \sum 1_{A_i}X_i \right)=\sum 1_{A_i}\alpha\left( 1_{A_i}X_i \right)=\sum 1_{A_i}\alpha\left( X_i \right)\geq \sum 1_{A_i}m_i=m,
			\end{equation*}
			so that $(X,m)\in \mathcal{A}_{\alpha}$, which shows the joint $\sigma$-stability of $\cA_\alpha$.

            Finally we prove the left-continuity of $\mathcal{A}_\alpha$.
            Let now $m\in \bar{L}^0$ and $B:=B(m)=\set{m>-\infty}$.
            By locality of $\alpha$, it follows that
            \begin{align*}
                \bigcap_{\substack{n<m\text{ on }B \\ n=-\infty \text{ on } B^c}  }\mathcal{A}^{n}_\alpha &= \bigcap_{n<m \text{ on }B} \mathcal{A}^{1_Bn-1_{B^c}\infty}=\Set{X \in \mathcal{X}: \alpha(X)\geq 1_B n- 1_{B^c}\infty \text{ for all } n<m \text{ on }B}\\
                &=\Set{X \in \mathcal{X}: \alpha(X)\geq 1_Bm-1_{B^c}\infty=m}=\mathcal{A}^m_\alpha
            \end{align*}
            By joint locality of $\mathcal{A}_\alpha$ and the fact that $\mathcal{A}_\alpha^{-\infty}=\set{X \in \mathcal{X}: \alpha(X)\geq -\infty}=\mathcal{X}$, it follows that
            \begin{equation*}
                1_B\bigcap_{\substack{n<m\text{ on }B \\ n=-\infty \text{ on } B^c}  }\mathcal{A}^{n}_\alpha+1_{B^c}\mathcal{X}=1_B\mathcal{A}^m_\alpha+1_{B^c}\mathcal{A}^{-\infty}=\mathcal{A}^m
            \end{equation*}
            Hence, left-continuity of $\cA_\alpha$ is proved.

		\item Conversely, assume $\mathcal{A}=(\mathcal{A}^m)_{m\in\bar{L}^0}$ to be an acceptance family, and consider $\alpha_{\mathcal{A}}$ defined as in \eqref{eq:AccIndandARF02}. First we prove that $\alpha_\cA$ is monotone.
			Take $X,Y\in\mathcal{X}$ such that $X\geqc Y$.
			By monotonicity of $\mathcal{A}$, if $Y\in\mathcal{A}^m$, then $X\in\mathcal{A}^m$.
			Hence $\set{ m\in\bar{L}^0 \, : \, Y\in\mathcal{A}^m } \subseteq \set{m\in\bar{L}^0 \, : \, X\in\mathcal{A}^m}$.
			Taking $\esssup$ of both sides  in the last inclusion, the monotonicity of $\alpha_{\mathcal{A}}$ follows.

			Next we will show that $\alpha_{\mathcal{A}}$ is quasiconcave.
			In order to do this, we consider $X,Y\in\mathcal{X}$ and we let $m, n \in\bar{L}^0$ be such that $X\in\mathcal{A}^m$ and $Y\in\mathcal{A}^n$.
			Such $m,n$ exist, since by the left-continuity of $\mathcal{A}$ we have $\cA^{-\infty}=\cX$.
			Next, we set $\widetilde{m}=m\wedge n,$ and from the decreasing property of $\mathcal{A}$ we conclude that $X,Y\in\mathcal{A}^{\widetilde{m}}$.
			Now, we take $\lambda \in L^0$ such that $0\leq \lambda\leq 1$.
			By convexity of $\mathcal{A}$, we get that the convex combination $Z:=\lambda X +(1-\lambda)Y\in\mathcal{A}^{\widetilde{m}}$, and hence $\alpha_{\mathcal{A}}(Z)\geq \widetilde{m}$.
			Consequently,
			$$
			\alpha_{\mathcal{A}}(Z)\geq \esssup \set{m\in \bar{L}^0 \, : \,  X\in\cA^m} \wedge  \esssup\set{n\in \bar{L}^0 \, : \,  Y\in\cA^n}.
			$$
			Thus, we conclude that $\alpha_{\mathcal{A}}(Z)\geq \alpha_{\mathcal{A}}(X) \wedge \alpha_{\mathcal{A}}(Y)$, which proves quasiconcavity of $\alpha_\cA$.

			It remains to prove locality of $\alpha_\cA$.
			For this, let $A\in\mathscr{G}$, $X\in\mathcal{X}$, and consider $m\in\bar{L}^0$ such that $X\in\mathcal{A}^{m}$.
			Again, such  $m$ exists since by the left-continuity of $\mathcal{A}$ we have $\mathcal{A}^{-\infty}=\mathcal{X}$.
			From  $\sigma$-stability of $\cA$, and from the fact that $0\in\cA^{-\infty}$, we have $1_{A}X \in\cA^{1_Am - 1_{A^c}\infty},$ which implies that $\alpha_{\cA}(1_AX)\geq 1_A m - 1_{A^c}\infty$.
			Hence, $1_A\alpha_{\cA}(1_AX)\geq 1_Am$, and thus, taking the essential supremum with respect to $m$ in this inequality, we get
			\begin{equation}\label{eq:cacy}
				1_A \alpha_\cA(1_A X)  \geq 1_A\alpha_{\cA}(X).
			\end{equation}

			Now, let $n\in \bar{L}^0$ such that  $1_AX\in\cA^{n}$. Since $X\in\cA^{-\infty}$, we get by $\sigma$-stability of $\mathcal{A}$, see Remark~\ref{rem:ss},
			$$
			X = 1_A(1_AX)+1_{A^c}X \in \cA^{1_An - 1_{A^c} \infty}.
			$$
			This implies that $\alpha_{\cA}(X) \geq 1_An - 1_{A^c}\infty$, and consequently $1_A\alpha_{\cA}(X)\geq 1_A n$. Taking the essential supremum with respect to $n$ in the last inequality, we get $1_A\alpha_{\cA}(X) \geq 1_A \alpha(1_AX)$, which, jointly with \eqref{eq:cacy}, demonstrates locality of $\alpha_\cA$.

			Thus $\alpha_{\mathcal{A}}$ is a conditional assessment index.

		\item We finally prove the last statement of Theorem \ref{th:AccIandARF}.
			Assume that $\alpha$ is a conditional assessment index.
			Then,  $\mathcal{A}_{\alpha}$ is a conditional risk acceptance family, and therefore $\alpha_{\mathcal{A}_{\alpha}}$ is a conditional assessment index.
			Note, that for any $X\in \mathcal{X}$ we have
			\begin{equation*}
				\alpha_{\mathcal{A}_{\alpha}}(X)=\esssup \Set{m\in \bar{L}^0\,:\, X\in \cA_\alpha^m}=\esssup \Set{m\in \bar{ L}^0\,:\, \alpha(X)\geq m}=\alpha(X),
			\end{equation*}
			and so $\alpha=\alpha_{\mathcal{A}_{\alpha}}$.

			Assume now that $\mathcal{A}$ is a conditional risk acceptance family.
			We will show that $\mathcal{A}_{\alpha_{\mathcal{A}}}^m=\mathcal{A}^m$ for any $m\in \bar L^0$, from which we deduce that $\mathcal{A}_{\alpha_{\mathcal{A}}}=\mathcal{A}$.
			If $m=-\infty$, then $\mathcal{A}_{\alpha_\mathcal{A}}^{-\infty} = \set{ X\in\mathcal{X} : \alpha_{\mathcal{A}}\left(X\right)\geq -\infty} = \mathcal{X}$, and by the left-continuity of $\mathcal{A}$ we get that $\mathcal{A}_{\alpha_{\mathcal{A}}}^{-\infty}=\mathcal{A}^{-\infty}$.
			Next, assume that $m>-\infty$.
			Given $\varepsilon\in L^0_{++}$, we claim that $\alpha_{\mathcal{A}}(X)\geq m$ implies that $X\in\mathcal{A}^{m-\varepsilon}$.
			Indeed, $\mathcal{A}$ being jointly $\sigma$-stable, $\set{n\in \bar{L}^0 : X\in\mathcal{A}^n}$ is upward directed.
			Hence, there exists an increasing sequence $(n_i)\subseteq\set{n\in\bar{L}^0 : X\in\mathcal{A}^n}$, such that $n_i\uparrow \alpha_{\mathcal{A}}(X)$.
			Let $A_i:=\set{n_i\geq m-\varepsilon}$ and $B_i := A_i\setminus A_{i-1}$, for $i\in\bN$, and put $B_0:=A_0$.
			Then $[B_i]\subseteq\mathscr{G}$, and $X\in\mathcal{A}^{n_i}$ for every $i\in\bN$.
			By $\sigma$-stability of $\mathcal{A}$, we get that $\sum 1_{B_i}(X,n_i)=(X,\sum 1_{B_i}n_i)\in \mathcal{A}$.
			However, by construction, $\widetilde{m} := \sum 1_{B_i}n_i \geq m-\varepsilon$, and thus $X\in\mathcal{A}^{m-\varepsilon}$.
			Consequently, we deduce
			\begin{align*}
				\mathcal{A}_{\alpha_{\mathcal{A}}}^{m} & = \Set{ X\in \mathcal{X} :  \alpha_{\mathcal{A}}(X) \geq m}  = \Set{ X\in\mathcal{X} : \esssup\set{ n\in \bar{L}^0 : X\in\mathcal{A}^n}\geq m} \\
				& = \Set{ X\in\mathcal{X} : X\in\cA^{m-\varepsilon} \textrm{ for all } \varepsilon\in L^0_{++} }  = \bigcap\limits_{n<m} \mathcal{A}^n.
			\end{align*}
			Finally, since in view of the left-continuity of $\mathcal{A}$ we have  $\cap_{n<m} \mathcal{A}^n= \mathcal{A}^m$, see Remark~\ref{rem:ss}, we obtain that $\mathcal{A}_{\alpha_{\mathcal{A}}}^{m} = \mathcal{A}^m$.
			For the general case $m\in\bar{L}^0$, we write $m=-1_{A^c}\infty + 1_{A}{m}$, with $A=\set{m>-\infty}$, and using  $\sigma$-stability we conclude that $\mathcal{A}_{\alpha_{\mathcal{A}}}^m=\mathcal{A}^m.$
			Thus, $\mathcal{A}_{\alpha_{\mathcal{A}}}=\mathcal{A}$.
	\end{enumerate}

\end{proof}

\subsection{Robust representation}\label{sec022}
In this subsection we prove a robust representation theorem for conditional assessment indices, which is one of the main results of this paper.
From now on we suppose that $\mathcal{X}$ is a conditional locally convex topological module.\footnote{See \cite{FilipovicKupperVogelpoth2009} and Appendix~\ref{appendix} for background on conditional modules.}
We further suppose that the cone $\mathcal{K}$ is closed, and we define the associated \emph{polar cone} by
\begin{equation*}
	\mathcal{K}^\circ:=\Set{X^\ast \in \mathcal{X}^\ast \, : \, \langle X^\ast,X\rangle \geq 0\text{ for all }X \in \mathcal{K}}.
\end{equation*}

We will now introduce two concepts, which are pivotal for our studies.

\begin{definition}\label{def:minRiskMeasure}
	A \emph{conditional risk function} is a function $R:\mathcal{K}^\circ \times \bar{L}^0\to \bar{L}^0$ such that
	\begin{enumerate}[label=(\roman*)]
		\item \label{cond:rm00} it is jointly local;

		\item \label{cond:rm01} the map	$s\mapsto R(X^\ast,s)$ is increasing and right-continuous for any $X^\ast \in \mathcal{K}^\circ$.
	\end{enumerate}
	A conditional risk function $R$ is called \emph{minimal} if
	\begin{enumerate}[label=(\roman*)]
			\setcounter{enumi}{2}
		\item \label{cond:rm02} it is jointly quasiconvex, and $R(\lambda X^\ast,s)=R\left(X^\ast,s/\lambda\right)$ for all $\lambda\in L^0_{++}$;
		\item \label{cond:rm03} it has a uniform asymptotic maximum, which means
			\begin{equation*}
				\esssup_{s \in L^0} R\left( X^\ast,s \right)=\esssup_{s \in L^0} R\left( Y^\ast,s \right),
			\end{equation*}
			for any $X^\ast,Y^\ast \in \mathcal{K}^\circ$;

\item \label{cond:rm04} the left-continuous version (in the second argument) $R^-\left(X^\ast,s \right)$  is jointly lower semicontinuous.
	\end{enumerate}
	The set of all conditional minimal risk functions is denoted by $\mathcal{R}^{\min}$.
\end{definition}
\begin{remark}
	Note that Condition \ref{cond:rm03} is equivalent to
	\begin{enumerate}[label=(\roman*${}^\prime$)]
			\setcounter{enumi}{3}
		\item\label{cond:rm03bis} it has a uniform asymptotic maximum, which means
			\begin{equation*}
				R^-\left( X^\ast,\infty \right)= R^-\left( Y^\ast,\infty \right),
			\end{equation*}
			for any $X^\ast,Y^\ast \in \mathcal{K}^\circ$;
	\end{enumerate}
\end{remark}

\begin{definition}\label{def:ConMaxPen}
	A \emph{conditional maximal penalty function} is a function $\pi :\mathcal{K}^\circ \times \bar{L}^0\to \bar{L}^0$ such that
	\begin{enumerate}[label=(\alph*)]
		\item \label{cond:pen00} it is jointly local;
		\item \label{cond:pen01} the map $m\mapsto \pi\left( X^\ast,m \right)$ is increasing and right-continuous for any $X^\ast \in \mathcal{K}^\circ$;
		\item \label{cond:pen02} it is positive homogeneous in the first argument\footnote{$\pi\left(\lambda X^\ast,m  \right)=\lambda\pi\left( X^\ast,m \right)$ for all $\lambda \in L^0_{++}$ and $X^\ast,m \in \mathcal{K}^\circ,\bar{L}^0$.} and concave in the first argument;
		\item \label{cond:pen03} it is  maximal invariant, that is, if $\pi\left( X^\ast,m \right)=\infty$ for some $X^\ast \in \mathcal{K}^\circ$ and $m\in\bar{L}^0$, then $\pi\left( Y^\ast,m \right)=\infty$ for all $Y^\ast \in \mathcal{K}^\circ$;

		\item \label{cond:pen04} it is upper semicontinuous in the first argument.
	\end{enumerate}
	The set of all conditional maximal penalty functions is denoted by $\mathcal{P}^{\max}$.	
\end{definition}

\begin{proposition}\label{prop:inverse}
	The set of conditional minimal risk functions $R \in \mathcal{R}^{\min}$ and the set of conditional maximal penalty functions $\pi \in \mathcal{P}^{\max}$ are related in the following manner
	\begin{align}
		\pi^{(-1,r)}\left( X^\ast,s\right)& \in \mathcal{R}^{\min},
		\label{eq:piinverse}\\
		R^{(-1,r)} \left( X^\ast,m \right)&\in \mathcal{P}^{\max},
		\label{eq:Rinverse}
	\end{align}
	where $\pi^{(-1,r)}\left( X^\ast,s\right)$ and $R^{(-1,r)} \left( X^\ast,m \right)$ denote the right-inverse\footnote{For further details apply Definition \ref{defn:inverse} for the second argument of $\pi$ and $R$, respectively.} in the second argument for fixed $X^\ast \in \mathcal{X}^\ast$. Moreover, the relationship is one-to-one.
\end{proposition}
\noindent The proof of this proposition is deferred to the Appendix \ref{appendix:02}.

\begin{proposition}\label{thm:01}
	Let $\mathcal{C}\subseteq \mathcal{X}$ be a closed, convex, monotone and $\sigma$-stable set.
	Then, there exists a unique local function $\pi:\mathcal{K}^\circ \to \bar{L}^0$ such that it is
	\begin{enumerate}[label=(\alph*)]
		\item \label{cond:pen02bis}  positive homogeneous and concave;
		\item \label{cond:pen03bis}  maximal invariant;
		\item \label{cond:pen04bis}  upper semicontinuous,
	\end{enumerate}
	and such that
	\begin{equation}
		X \in \mathcal{C}\quad \Longleftrightarrow \quad \langle X, X^\ast\rangle \geq \pi\left( X^\ast \right), \quad \text{for all }  X^\ast \in \mathcal{K}^{\circ}.
		\label{robust1}
	\end{equation}
	Moreover, this function is explicitly given by the relation
	\begin{equation*}
		\pi\left( X^\ast \right)=\chi_{\mathcal{C}}^\star\left( X^\ast \right):=\essinf_{X \in \mathcal{C}}\langle X^\ast,X\rangle,\quad \text{for all }  X^\ast \in \mathcal{K}^{\circ}.
	\end{equation*}
\end{proposition}
\noindent The proof of this proposition is also deferred to the Appendix~\ref{appendix:03}.

\noindent Finally, we are in the position to prove the main result of this section.
\begin{theorem}\label{thm:robrep} \mbox{}
	\begin{enumerate}[label=\textnormal{\textbf{(\roman*)}}]
		\item 	Let $\alpha:\mathcal{X}\to \bar{L}^0$ be an upper semicontinuous conditional assessment index.
			Then, $\alpha$ has the robust representation of the form
			\begin{equation}
				\alpha\left( X \right)=\essinf_{X^\ast \in \mathcal{K}^\circ} R\left(X^\ast,\langle X^\ast,X\rangle  \right),
				\label{eq:robrep}
			\end{equation}
			for a unique $R \in \mathcal{R}^{\min}$; 	
		\item For any conditional risk function $R$, the right hand-side of  \eqref{eq:robrep} defines an upper semi\-continuous conditional assessment index.
	\end{enumerate}
\end{theorem}
\begin{proof}
	\begin{enumerate}[label=\textnormal{\textbf{(\roman*})},fullwidth]
		\item According to Theorem~\ref{th:AccIandARF},
			\begin{equation}
				\alpha\left( X \right)=\esssup \Set{m \in \bar{L}^0 \, : \, X \in \mathcal{A}^{m}}, \quad X \in \mathcal{X},
				\label{eq:robrep01}
			\end{equation}
			where $\mathcal{A}=(\cA^m)_{m\in\bar{L}^0}$ is the corresponding conditional risk acceptance family in the sense of \eqref{eq:AccIndandARF01}.
			In particular, each of the sets $\cA^m, \ m\in\bar{L}^0$, is monotone, convex, and, in view of Remark~\ref{rem:ss}, it  is also $\sigma$-stable.
			In addition, since $\alpha$ is upper-semicontinuous then each set $\cA^m, \ m\in\bar{L}^0$ is closed.
			Thus, defining  $\pi:\mathcal{K}^\circ\times \bar{L}^0 \to \bar{L}^0$ by
			\begin{equation}\label{eq:defPiMax}
				\pi\left( X^\ast,m \right):=\essinf_{X \in \mathcal{A}^m}\langle X^\ast,X\rangle,
			\end{equation}
			we have by Proposition~\ref{thm:01} that $\pi$ satisfies properties \ref{cond:pen02}-\ref{cond:pen04} of Definition~\ref{def:ConMaxPen}.
			Moreover, from \eqref{robust1}, we conclude that
			\begin{equation} \label{eq:robrep01-1}
				X\in\cA^m \quad \Longleftrightarrow \quad  \langle X,X^\ast\rangle \geq \pi\left( X^\ast,m \right)\text{ for all }X^\ast \in \mathcal{K}^\circ,
			\end{equation}
			which in combination with \eqref{eq:robrep01} yields
			\begin{equation}
				\alpha\left( X \right)=\esssup \Set{m \in \bar{L}^0 \, : \, \langle X,X^\ast\rangle \geq \pi\left( X^\ast,m \right),\text{ for all }X^\ast \in \mathcal{K}^\circ}.
				\label{eq:robrep02}
			\end{equation}

Furthermore, $\pi$ fulfills \ref{cond:pen00} of Definition~\ref{def:ConMaxPen}.
			Indeed, let $X^\ast,Y^\ast \in \mathcal{K}^\circ$, $m,\widetilde{m}\in \bar{L}^0$ and $A \in \mathscr{G}$.
			Since, $\mathcal{A}$ is jointly $\sigma$-stable, it follows that $\mathcal{A}^{1_{A}m+1_{A^c}\widetilde{m}}=1_{A}\mathcal{A}^m+1_{A^c}\mathcal{A}^{\widetilde{m}}$.
			Hence
			\begin{align*}
				\pi\left( 1_{A}X^\ast+1_{A^c}Y^\ast,1_{A}m+1_{A^c}\widetilde{m} \right) &=\essinf_{\widetilde{X}\in \mathcal{A}^{1_{A}m+1_{A^c}\widetilde{m}}}\langle 1_{A}X^\ast+1_{A^c}Y^\ast,\widetilde{X}\rangle\\
				&=\essinf_{X\in \mathcal{A}^m, Y\in \mathcal{A}^{\widetilde{m}}}\langle 1_{A}X^\ast+1_{A^c}Y^\ast,1_{A}X+1_{A^c}Y\rangle\\
                &= 1_{A}\essinf_{X\in \mathcal{A}^m}\langle X^\ast,X\rangle +1_{A^c}\essinf_{Y \in \mathcal{A}^{\widetilde{m}}}\langle Y^\ast,Y\rangle\\
				&=1_{A}\pi\left( X^{\ast},m \right)+1_{A^c}\pi\left( Y^\ast,\widetilde{m} \right),
			\end{align*}
hence $\pi$ is jointly local.

			Since the map  $m \mapsto \pi\left( \cdot,m \right)$ is increasing\footnote{Due to the fact that $\cA$ is decreasing.},
			the left- and right-continuous version of it, say,  $\pi^-$ and  $\pi^+$ respectively, are given as in \eqref{eq:app01} and \eqref{eq:app02}.
			Moreover, it is rather clear that $\pi^+$ fulfills\footnote{In particular, notice that an essential infimum of
			a family of upper semicontinuous functions is an upper semicontinuous, and in view of  \eqref{eq:app02} $\pi^+$ is upper semicontinuous.} the conditions \ref{cond:pen00}-\ref{cond:pen04} of  Definition~\ref{def:ConMaxPen}, and thus $\pi^+ \in \mathcal{P}^{\max}$.

			Next we show that
			\begin{equation}\label{eq:Beta3}
				\alpha(X) = \beta^-\left( X \right)=\beta^+(X),  \quad X\in\cX,
			\end{equation}
			where 	
			\begin{align}
				\beta^-\left( X \right) & :=\esssup \Set{m \in \bar{L}^0 \, : \, \langle X,X^\ast\rangle
				\geq \pi^-\left( X^\ast,m \right)\text{ for all }X^\ast \in \mathcal{K}^\circ}, \\
				\beta^+\left( X \right)& :=  \esssup \Set{m \in \bar{L}^0 \, : \, \langle X,X^\ast\rangle \geq \pi^+\left( X^\ast,m \right)\text{ for all }X^\ast \in \mathcal{K}^\circ}.
				\label{eq:beta4}
			\end{align}
			Since $\pi^-\left( X^\ast,m \right)\leq \pi\left( X^\ast,m \right)\leq \pi^+\left( X^\ast,m \right)$ for all $X^\ast ,m \in \mathcal{K}^\circ\times \bar{L}^0$,
			it follows that
			\begin{align}\label{eq:bloedsinn01}
				\beta^-\left( X \right) \geq \alpha(X) \geq \beta^+\left( X \right), \quad X\in\cX.
			\end{align}
			If $\beta^-\left( X \right)$ is equal to $-\infty$ on some set $A \in \mathscr{G}_{+}$, then equality \eqref{eq:Beta3} holds true on $A$.
			Hence, using locality, it is enough to prove that \eqref{eq:Beta3}  holds true for  $\beta^-\left( X \right)>-\infty$.
			By the definition of $\beta^-$,  there exists an increasing sequence $(m^n)\subseteq L^0$ converging to $\beta^{-}\left( X \right)$, and such that $m^n<m^{n+1}<\beta^{-}\left( X \right)$.
			By the definition of the left-and right-continuous version of an increasing function, we get $\pi^+\left( X^\ast,m^n \right)\leq \pi^-\left( X^\ast,m^{n+1}\right)$, for all $X^\ast \in \mathcal{K}^\circ$, and all $n \in \mathbb{N}$.
			Hence, $m^n\leq \beta^+\left( X \right)$ for all $n\in\bN$, and therefore $\beta^+\left( X \right)\geq \beta^-\left( X \right)$. This, combined with \eqref{eq:bloedsinn01}, implies  \eqref{eq:Beta3}.

			Denote by $R$ the right-inverse of $\pi^+$.
			By Proposition~\ref{prop:leftrightinverse}, see Remark \ref{rem:inverseleftright}, we have that $R=(\pi^-)^{(-1,r)}$. Thus, by \eqref{eq:Beta3} and \eqref{C3} we conclude that
			\begin{equation*}
				\alpha\left( X \right)  =\esssup \Set{m \in \bar{L}^0:  R\left(X^\ast,\langle X,X^\ast\rangle\right) \geq m \text{ for all }X^\ast \in \mathcal{K}^\circ},
			\end{equation*}
			and, consequently,
			\begin{equation*}
				\alpha\left( X \right)		=\esssup\Set{m \in \bar{L}^0: \essinf_{X^\ast \in \mathcal{K}^\circ}R\left(X^\ast,\langle X,X^\ast\rangle\right) \geq m }=\essinf_{X^\ast \in \mathcal{K}^\circ}R\left( X^\ast,\langle X^\ast,X\rangle \right).
			\end{equation*}
			
			Finally, we show the uniqueness of $R\in\cR^{\min}$. Using Proposition~\ref{prop:inverse} and \eqref{eq:Beta3}, it is sufficient to show that
			\begin{equation}\label{eq:beta5}
				\alpha(X) = \esssup \Set{m \in \bar{L}^0 \, : \, \langle X,X^\ast\rangle \geq \widetilde{\pi}\left( X^\ast,m \right)\text{ for all }X^\ast \in \mathcal{K}^\circ}.
			\end{equation}
			holds true for a unique $\widetilde{\pi}\in\cP^{\max}$.
			We assume, that 	\eqref{eq:beta5} is satisfied for $\pi^i\in\mathcal{P}^{\max}, \ i=1,2$.
			For every $n\in\bar{L}^0$ and $i=1,2$, we consider the sets
			\begin{align}
				\mathcal{A}^{n,i} & :=\set{X \in \mathcal{X} \, : \, \langle X^\ast,X\rangle \geq \pi^i\left( X^\ast,n \right)\text{ for all }X^\ast \in \mathcal{K}^\circ}\nonumber \\
				& \ = \bigcap_{X^\ast \in \mathcal{K}^\circ} \set{X \in \mathcal{X}: \langle X^\ast,X \rangle \geq \pi^i(X^\ast,n)}. \label{eq:beta6}
			\end{align}
			For every $X^\ast \in \mathcal{K}^\circ, \ m\in\bar{L}^0$, the set $\set{X \in \mathcal{X}: \langle X^\ast,X \rangle \geq  m}$ is clearly closed, convex, and $\sigma$-stable and monotone.
			By \eqref{eq:beta6}, we conclude  that $\mathcal{A}^{n,i}$ are closed, convex, monotone and $\sigma$-stable, for every $n\in\bar{L}^0$ and $i=1,2$.
			Let $A=\set{m=\infty}$.
			By Proposition \ref{thm:01}, we have that, for $i=1,2$,
			\begin{equation} \label{eq:beta5-1}
				\pi^i\left( X^\ast,m \right)=\essinf_{\substack{n\geq m\\n>m\text{ on }A}}\pi^i\left( X^\ast,n \right)=\essinf_{\substack{n\geq m\\n>m\text{ on
						}A}}\essinf_{X \in \mathcal{A}^{n,i}} \langle X^\ast,X\rangle=\essinf_{X \in \bigcup\limits_{\substack{n\geq m\\n>m\text{ on }A}} \cA^{n,i}}\langle X^\ast ,X\rangle.
			\end{equation}
			on $A^c$.

			Next we will show that
			\begin{equation}\label{eq:beta7}
				\bigcup_{\substack{n\geq m\\n>m\text{ on }A}}\cA^{n,i} = \Set{X \in \mathcal{X}:\alpha(X)\geq m\text{ and }\alpha\left( X \right)>m\text{ on }A}, \quad i=1,2.
			\end{equation}
			If $X$ belongs the left hand side, then $X\in A^{n_0,i}$ for some $n_0\geq m$ with $n_0>m$ on $A$, and hence, by \eqref{eq:Beta3} together with \eqref{eq:beta4}, we get that $\alpha(X)\geq n_0$, and consequently we conclude that $X$ belongs to the right hand side.
		  	Conversely, if $\alpha\left( X \right)\geq m$ with $\alpha(X)>m$ on $A$, then by \eqref{eq:Beta3} together with \eqref{eq:beta4}, there exists $n_0 \geq m$ with $n_0>m$ on $A$ such that $\langle X^\ast,X\rangle \geq \pi^i\left( X^\ast,n_0 \right)$ for all $X^\ast \in \mathcal{K}^\circ$.
			Hence, $X\in A^{n_0,i}$ and therefore $X$ is in the left hand side of \eqref{eq:beta7}.
			Finally, \eqref{eq:beta5-1} combined with \eqref{eq:beta7} imply that $\pi^1=\pi^2$ on $A^c$.
			Since $\pi^i$ are right-continuous, $\pi^1=\pi^2=\infty$ on $A$, and thus $\pi^1=\pi^2$.

		\item If the function $R$ is a conditional risk function, i.e. it satisfies \ref{cond:rm00} and \ref{cond:rm01} from Definition~\ref{def:minRiskMeasure}, it follows immediately that for every $X^\ast\in\cK^\circ$, the function $R\left(X^\ast,\langle X^\ast,\cdot\rangle  \right)$ is local, quasiconcave, monotone, and upper-semicontinuous.  All these properties are preserved under $\essinf$, and this concludes the proof.

	\end{enumerate}

\end{proof}

\begin{remark}\label{rem:normalizedset}
	Similarly to  \citep{DrapeauKupper2010}, if there exists $\kappa \in \mathcal{K}$ such that $\langle X^\ast,\kappa \rangle>0$ for any $X^\ast \in \mathcal{K}^\circ$, the robust representation \eqref{eq:robrep} can be achieved on the normalized set
	\begin{equation}
		\mathcal{K}^\circ_{\kappa}:=\Set{X^\ast \in \mathcal{K}^\circ:\langle X^\ast, \pi\rangle=1},
		\label{}
	\end{equation}
	for a unique minimal risk function $R:\mathcal{K}^\circ_{\kappa}\times L^0\to \bar{L}^0$.
	In this case  the condition \ref{cond:rm02} from Definition~\ref{def:minRiskMeasure} is replaced by
	\begin{enumerate}[label=(\textit{iii})${}^\prime$]
		\item \label{cond:rm02bis} it is jointly quasiconvex.
	\end{enumerate}
\end{remark}

\noindent
Additional properties of $\alpha$ are shared by the corresponding dual minimal risk function, as stated in the following result.

\begin{proposition}
	An upper semicontinuous assessment index $\alpha$ is concave, positive homogeneous, scale invariant, or $\kappa$-cash additive if and only if the corresponding minimal risk function $R$ is convex, positive homogeneous, scale invariant or $\kappa$-cash additive in the second argument, respectively.
\end{proposition}
The proof is similar to that from the static case (cf. \citep{Drapeau2010,DrapeauKupper2010}), and we omit it here.

\subsection{Scale Invariant Conditional Assessment Indices}\label{sec:scaleInv}
In this section we specify how the robust representation looks like in the specific case of scaling invariance.
Note that the acceptance sets $\mathcal{A}^m,\ m\in L^0,$ corresponding to a scale invariant assessment index are closed and convex cones.
We denote their polar sets as
\begin{equation}
	\mathcal{A}^{m,\circ}:=\Set{X^\ast \in \mathcal{X}^\ast:\langle X^\ast,X\rangle \geq 0\text{ for all }X\in \mathcal{A}^m},\quad m \in \bar{L}^0.
	\label{eq:ampolar}
\end{equation}

\begin{proposition}\label{prop:ScaleInvGeneral}
	Let $\alpha: \mathcal{X}\to \bar{L}^0$ be an upper semicontinuous scale invariant conditional assessment index.
	Then, the unique conditional minimal risk function $R \in \mathcal{R}_{\max}$ from the representation \eqref{eq:robrep}  has the form
	\begin{equation}\label{eq:scaleR}
		R\left( X^\ast,s \right)=
		\begin{cases}
			-\infty&\text{on }\set{s=-\infty}\\
			\essinf\Set{m \in \bar{L}^0:X^\ast \in \mathcal{A}^{m,\circ}}&\text{on }\set{-\infty<s<0}\\
			+\infty&\text{on }\set{s\geq 0}
		\end{cases}
		,\quad X^\ast \in \mathcal{K}^\circ\text{ and }s \in \bar{L}^0.
	\end{equation}
\end{proposition}

\begin{proof}
Similar to Theorem~\ref{thm:robrep}.(i), we consider the function  
\begin{equation}\label{eq:scaleInv2}
		\pi\left( X^\ast,m \right):=\essinf_{X \in \mathcal{A}^m}\langle X^\ast,X\rangle=\chi^\star_{\mathcal{A}^m}\left( X^\ast \right),
	\end{equation}
 where the last equality follows as in \eqref{eq:myblabla02}.
	Since $\mathcal{A}^{m}$ is a cone, it follows that
	\begin{equation}\label{eq:pimaxpolar}
		\chi^{\star}_{\mathcal{A}^m}=\chi_{\mathcal{A}^{m,\circ}},
	\end{equation}
	for any $m \in \bar{L}^0$.\footnote{Note that this states the conditional version of the Bipolar Theorem.}
	Indeed, by definition, $X^\ast \in \mathcal{A}^{m,\circ}$ if and only if $\langle X^\ast,X\rangle\geq 0$ for every $X \in \mathcal{A}^m$.
	Using the fact that $\mathcal{A}^m$ is a cone, we scale $X$ with $\lambda \in L^0_{++}$ converging to $0$ in the essential infimum \eqref{eq:scaleInv2} 
	It follows that $X^\ast \in \mathcal{A}^{m,\circ}$ if and only if $\chi^{\star}_{\mathcal{A}^m}\left( X^\ast \right)=0$.
	On the other hand, if $1_{B}X^\ast \not \in 1_{B}\mathcal{A}^{m,\circ}$ for every $B \in \mathscr{G}_{+}$, it follows by definition of $\mathcal{A}^{m,\circ}$ that there exists $X \in \mathcal{A}^{m,\circ}$ such that $\langle X^\ast,X\rangle<0$.
	Scaling with $\lambda \in L^0_{++}$ tending to $\infty$, it follows that $1_{B}X^\ast \not \in 1_{B}\mathcal{A}^{m,\circ}$ for every $B \in \mathscr{G}_{+}$ if and only if $\pi\left( X^\ast,m \right)=-\infty$.
	By locality, and definition of $\chi_{\mathcal{A}^{m,\circ}}$, we therefore deduce that equation \eqref{eq:pimaxpolar} holds.

	Finally, we need to show that $R$ given by $\eqref{eq:scaleR}$ is the conditional right-inverse of $\pi$ in the second argument.
It holds that  $\chi_{\mathcal{A}^{m,\circ}}$ takes only $0$ and $\infty$ as values.
	For $X^\ast=0$, it clearly holds $R(0,s)=\infty$ on $\set{s\geq 0}$ and $-\infty$ on $\set{s<0}$ which corresponds to Relation \eqref{eq:scaleR}.
	On the other hand, if $1_A X^\ast\neq 0 $ for every $A \in \mathcal{G}_{+}$, it follows that $\essinf_{m \in \bar{L}^0} \chi_{\mathcal{A}^{m,\circ}}(X^\ast)=\chi_{\mathcal{X}^\circ}(X^\ast)=\chi_{\set{0}}(X^\ast)=-\infty$.
	Hence applying the definition of the right inverse, it follows that
	\begin{align*}
		R\left( X^\ast,s \right)&=-1_{\set{s=-\infty}}\infty+1_{\set{s\geq 0}}\infty+1_{\set{-\infty<s<0}}\essinf\Set{m \in \bar{L}^0:\chi_{\mathcal{A}^{m,\circ}}(X^\ast)>s}\\
		&=-1_{\set{s=-\infty}}\infty+1_{\set{s\geq 0}}\infty+1_{\set{-\infty<s<0}}\essinf\Set{m \in \bar{L}^0:X^\ast\in \mathcal{A}^{m,\circ}}.
	\end{align*}
	Using stability for the general $X^\ast \in \mathcal{K}^\circ$ yields the representation \ref{eq:scaleR}.
\end{proof}

\subsection{Certainty Equivalent}\label{sec024}
In \citet{CheriditoKupper2009}, a concept of certainty equivalent was studied in the context of risk measures.
Here, we carry out an analogous study with regard to conditional assessment  indices. In Section~\ref{sec04}, we will make a crucial use of the concept of certainty equivalent in studying the (strong) time consistency of assessment indices for processes.
Throughout this section we fix $\kappa\in\mathcal{K}\setminus \set{0}$.

\begin{definition}\label{def:certEquivalent}
	A \emph{$\kappa$-conditional certainty equivalent} of a conditional assessment index $\alpha$ is a local functional $C:\mathcal{X}\to L^0$ such that
	\begin{equation}
		\alpha(C(X)\kappa)=\alpha(X), \quad X \in \mathcal{X}.
		\label{eq:certEquivalent0}
	\end{equation}
\end{definition}
\noindent A natural candidate for the conditional certainty equivalent of a conditional assessment index $\alpha$ is given by
\begin{equation} \label{eq:certEquivalent}
	C(X):= \essinf \Set{ m \in L^0:\alpha(m\kappa) \geq \alpha(X)},\quad X \in \mathcal{X}.
\end{equation}

\begin{remark}\label{rem:CerEqNotGeneral}
However, in general, definition \eqref{eq:certEquivalent}, even though natural, may not produce a valid certainty equivalent.
In particular, if $\alpha$ is a scale invariant assessment index, then $C(X)$ defined as in \eqref{eq:certEquivalent}, will take values only $0$ and $-\infty$, and \eqref{eq:certEquivalent0} will not be satisfied, in general. Indeed, for simplicity assume that $\cK=L^0_+$ and $\kappa=1$, and let $C$ be defined as in \eqref{eq:certEquivalent}.
For sufficiently large $m>0$, we have that $m\geq X$, and by monotonicity of $\alpha$, we deduce that $\alpha(m)\geq \alpha(X)$. Hence, using scale invariance of $\alpha$, we conclude that $C(X)\leq 0$, and consequently
$$
C(X) = \essinf \Set{ m \in L^0: m\leq 0, \textrm{ and } \alpha(m) \geq \alpha(X)},\quad X \in \mathcal{X}.
$$
Using scale invariance of $\alpha$ again, we conclude that $C(X)$ will take values only $0$ and $-\infty$.

With  \eqref{eq:certEquivalent} in mind, we thus need to find sufficient conditions on index $\alpha$ ensuring that \eqref{eq:certEquivalent} indeed defines a certainty equivalent.
\end{remark}

\begin{definition}
	A conditional assessment index $\alpha$ is
	\begin{itemize}
		\item $\kappa$\emph{-bounded}, if for any $X \in \mathcal{X}$, there exist $m_1,m_2 \in L^0$ satisfying
			\begin{equation}\label{l0bounded}
				\alpha(m_1\kappa) < \alpha(X) \leq \alpha(m_2\kappa).		
			\end{equation}
		\item $\kappa$\emph{-strictly increasing}, if $\alpha(m\kappa) > \alpha(m'\kappa)$ on $A$, whenever $m,m'\in L^0$ and $m>m'$ on $A$.
		\item $\kappa$\emph{-sensitive}, if for $m\in L^0$ and $Y\in \mathcal{X}$ with $\alpha(m\kappa) > \alpha(Y)$ on some $A \in \mathscr{G}$, there exists an $\varepsilon\in L^0_{+}$ with $\varepsilon>0$ on $A$,   such that
			\begin{equation*}
				\alpha((m-\varepsilon)\kappa )  \geq \alpha(Y),\quad\text{ on }A.
			\end{equation*}
	\end{itemize}
\end{definition}

\begin{proposition}\label{prop:relation:alpha:C:bis}
	Let $\alpha:\mathcal{X}\to \bar{L}^0$ be a $\kappa$-sensitive and $\kappa$-bounded upper semicontinuous conditional assessment index.
	Then, $C$ defined as in \eqref{eq:certEquivalent} is a $\kappa$-conditional certainty equivalent and
	\begin{equation}\label{eq:relation:alpha:C1}
		\alpha\left( X \right)\geq \alpha\left( Y \right)\quad\Longleftrightarrow\quad C\left( X \right)\geq C\left( Y \right), \quad X,Y\in \cX.		
	\end{equation}
	In this case, $C$ is itself a $\kappa$-sensitive and $\kappa$-bounded conditional assessment index.

	If in addition $\alpha$ is $\kappa$-strictly increasing, then \eqref{eq:certEquivalent} is upper semicontinuous, and the unique $\kappa$-conditional certainty equivalent of $\alpha$.
\end{proposition}
\begin{remark}
	Relation \eqref{eq:relation:alpha:C1} shows that $C$ and $\alpha$ reproduce the same ranking, so they are equivalent in this sense.
	Note that the functional defined in \eqref{eq:certEquivalent} satisfies the following property
	\begin{equation}
		C\left( C\left( X \right)\kappa \right)=C\left( X \right),\quad X \in \mathcal{X},
	\end{equation}
which means that $C$ is a certainty equivalent of itself.
\end{remark}

\begin{proof}
Let $C$ be defined as in \eqref{eq:certEquivalent}. Consequently, \eqref{l0bounded} implies that $C$ takes values in $L^0$.
	Next we will show that $C$ satisfies \eqref{eq:certEquivalent0}.
	By locality of $\alpha$ the set $\mathcal{C}(X):=\set{m\in L^0  : \alpha(m\kappa)\geq \alpha(X)}$ is downward directed.
	Hence, there exits a decreasing sequence $(m_n)\subseteq\mathcal{C}(X)$ converging to $C(X)$ $P$-almost surely.
	Upper semicontinuity of $\alpha$ implies
	\begin{equation}\label{eq:certEquivalent2}
		\alpha\left(C\left(X\right)\kappa\right) = \alpha\left(\lim_n m_n\kappa\right)\geq \esslimsup_n \alpha\left( m_n \kappa \right)\geq \alpha(X).
	\end{equation}
	Suppose now that $\alpha\left( C\left( X \right)\kappa \right)>\alpha\left( X \right)$ on some $A \in \mathscr{G}_{+}$.
	By $\kappa$-sensitivity of $\alpha$ it follows that $\alpha((C(X)-\varepsilon)\kappa)\geq \alpha(X)$ on $A$, for some $\varepsilon>0$ on $A$.
	Take $\varepsilon=0$ on $A^c$, and by locality of $\alpha$ and \eqref{eq:certEquivalent2}, we get that $\alpha(C(X)\kappa -\varepsilon)\geq \alpha(X)$.
	Hence, $C(X)-\varepsilon\in\mathcal{C}(X)$, so that $C(X)-\varepsilon\geq C(X)$, which is a contradiction.
Next, let us prove that $C$ is local.
By the definition of $C$, and locality of $\alpha$, we have
\begin{align*}
	C(1_A X+1_{A^c}Y)& = \essinf\Set{m \in L^0: 1_A\alpha(m\kappa)+1_{A^c}\alpha(m\kappa)\geq 1_A\alpha(X)+1_{A^c}\alpha(Y)}\\
	& = \essinf\{1_Am_1+1_{A^c}m_2 \in L^0: 1_A\alpha\left(\left(1_{A}m_1+1_{A^c}n_1\right)\kappa\right)\geq 1_A\alpha(X), \\
	& \qquad \qquad\qquad 1_{A^c}\alpha\left((1_An_2+1_{A^c}m_2)\kappa\right)\geq 1_{A^c}\alpha(Y), \text{ where }n_1,n_2 \in L^0\}\\
	&=  1_A\essinf\Set{1_Am_1+1_{A^c}n_1 \in L^0:1_A\alpha\left((1_{A}m_1+1_{A^c}n_1)\kappa\right)\geq 1_A\alpha(X)}\\
	&\quad + 1_{A^c}\essinf\set{1_{A^c}m_2+1_An_2 \in L^0:1_{A^c}\alpha\left((1_{A}n_2+1_{A^c}m_2)\kappa\right)\geq 1_{A^c}\alpha(Y)}\\
	&=1_{A}C(X)+1_{A^c}C(Y)
\end{align*}
where in the second equality we used the $\kappa$-boundedness assumption to ensure the existence of $n_1,n_2\in L^0$, such that $1_{A^c}\alpha(1_am_1+1_{A^c}n_1)\kappa)\geq 1_{A^c}\alpha(X)$ and $1_{A}\alpha((1_{A}n_2+1_{A^c}m_2)\kappa)\geq 1_{A}\alpha(Y)$.
Hence, $C$ is local. Thus, $C$ is a $\kappa$-conditional certainty equivalent.

	Next, we will show that \eqref{eq:relation:alpha:C1} is fulfilled.
	Clearly, $\alpha(X)\geq \alpha(Y)$ implies $C(X)\geq C(Y)$.
	Suppose that $\alpha(X)\geq \alpha(Y)$, and  $\alpha(X)>\alpha(Y)$ on some $A \in \mathscr{G}_{+}$.
	Since $C$ is a $\kappa$-conditional certainty equivalent of $\alpha$, it follows that $\alpha(C(X)\kappa) > \alpha(Y)$ on $A$.
	By similar arguments as above, since $\alpha$ is $\kappa$-sensitive there exists $\varepsilon\in L^0_+$ with $\varepsilon>0$ on $A$, and $\varepsilon=0$ on $A^c$, such that $\alpha((C(X)-\varepsilon)\kappa)\geq \alpha(Y)$.
	Hence, $C(X)-\varepsilon\in\mathcal{C}(Y)$, and thus $C(X)-\varepsilon \geq C(Y)$, which implies that $C(X)>C(Y)$ on $A$.
	Thus \eqref{eq:relation:alpha:C1} is established.

	Note that by means of relation \eqref{eq:relation:alpha:C1}, $\alpha$ and $C$ define the same conditional preference order on $\mathcal{X}$.
	Thus, $C$ is itself a conditional assessment index.\footnote{Both monotonicity and quasiconcavity of $C$ follow from corresponding properties of $\alpha$ and relation \eqref{eq:relation:alpha:C1}.}
	Also by \eqref{eq:relation:alpha:C1} we conclude that $\alpha$ being $\kappa$-bounded implies that $C$ is $\kappa$-bounded.
	Next we will show that $C$ is $\kappa$-sensitive.
	Take $m \in L^0$ and $X \in \mathcal{X}$ such that $C(m\kappa)>C(X)$ on some set $A \in \mathscr{G}$.
	Using locality of $\alpha$ and $C$, and by \eqref{eq:relation:alpha:C1}, it follows that $\alpha(m\kappa)>\alpha(X)$ on $A$.
	Hence, by $\kappa$-sensitivity of $\alpha$,  there exists $\varepsilon \in L^0_{+}$ with $\varepsilon>0$ on $A$ such that $\alpha( (m-\varepsilon )\kappa)\geq \alpha(X)$ on $A$.
	Again, using locality and \eqref{eq:relation:alpha:C1},  we conclude that  $C( (m-\varepsilon)\kappa)\geq C(X)$ on $A$.
	This shows that $C$ is $\kappa$-sensitive.

Let us assume that $\alpha$ in addition is $\kappa$-strictly increasing.
	We claim that $C(m\kappa)=m, \ m\in L^0$.
	Indeed, by definition \ref{eq:certEquivalent}, we have that $C(m\kappa)\leq m$.
	Suppose that  $C(m\kappa)<m$ on some set $A \in \mathscr{G}_{+}$.
    Since $\alpha$ is $\kappa$-strictly increasing, it follows that $\alpha\left( C(m\kappa) \right)<\alpha(m\kappa)$ on $A$.
	However, $\alpha(C(m\kappa))=\alpha(m\kappa)$ which is a contradiction.
	Next we will show that any  $\kappa$-certainty certainty  equivalent $\widetilde{C}$ of $\alpha$ is equal to $C$.
	Given $X\in \mathcal{X}$, we note that $\widetilde{C}(X)\in \mathcal{C}(X)$, and hence $C(X)\leq \widetilde{C}(X)$.
	Suppose that $C(X)<\widetilde{C}(X)$ on some $A$.
	Since $\alpha$ is $\kappa$-strictly increasing and local, it follows that $\alpha(X)=\alpha(C(X)\kappa)<\alpha(\widetilde{C}(X)\kappa)=\alpha(X)$ on $A$ which is a contradiction.
	Thus $\widetilde{C}=C$.
	Finally, it remains to show that $C$ is upper semicontinuous.
	For a given $m \in L^0$, using the statements proved above, we deduce that
	\begin{equation*}
		\Set{X \in \mathcal{X}:C(X)\geq m}=\Set{X \in \mathcal{X}: C(X)\geq C(m\kappa)}=\Set{X \in \mathcal{X}: \alpha(X)\geq \alpha(m\kappa)}.
	\end{equation*}
	The latter set is closed since $\alpha$ is upper semicontinuous, and hence, the upper level sets of $C$ are also closed, and thus $C$ is upper semicontinuous.
\end{proof}

\begin{remark}
	 Note that if $\alpha$ is a $\kappa$-bounded and $\kappa$-cash additive acceptability index, then, up to a translation by $\alpha(0)$, $\alpha$ is a certainty equivalent of itself.
	 In other terms $C(X):=\alpha(X)-\alpha(0)$ is a certainty equivalent of $\alpha$.
	 Indeed, $\kappa$-boundedness and $\kappa$-cash additive implie that $\alpha$ only takes values in $L^0$, and thus $C$ also takes values only in $L^0$.
	 Moreover, since $\alpha(m\kappa)=\alpha(0)+m$, we have that $\alpha(C(X)\kappa)=\alpha(0)+C(X)=\alpha(X)$.
\end{remark}

\section{Assessment Indices for Stochastic Processes}\label{sec03}

We will now apply the theory developed in Section \ref{sec:AIGEneral} to study of assessment indices for discrete time, real valued random processes.

\subsection{Conditional Assessment Indices for Stochastic  Processes}\label{sec031}
In this section we follow the approach and notations for stochastic processes introduced by \citet{AcciaioFollmerPenner2010}.
Given a time horizon $T\in \bN$, let $(\Omega,\mathscr{F},P)$ be a probability space with a filtration $(\mathscr{F}_s)$ where $s$ is in $\set{0,\ldots,T}$.
Given $t \in \set{0,\ldots,T}$, we denote by $\mathscr{O}^t$ the optional $\sigma$-algebra up to time $t$ on the product space $\widetilde{\Omega}:=\Omega \times \set{0,\ldots,T}$, which is equal to
\begin{equation}\label{eq:ot}
	\mathscr{O}^t=\sigma\left(\Set{A_s\times\set{s},A_t\times\set{t,\ldots,T}:s<t,A_s \in \mathscr{F}_s\text{ and }A_t \in \mathscr{F}_t}\right).
\end{equation}

We define $\mathscr{O}:=\mathscr{O}^T$.
On $\widetilde{\Omega}$ we denote by $\widetilde{P}$ a probability  measure, which is defined by the expectation
\begin{equation*}
	E_{\widetilde P}[X] := E_P\left[ \sum_{s=0}^T X_s\mu_s \right],
\end{equation*}
where $\mu$ is some adapted process such that $\sum_{s=0}^T \mu_s=1$ and $\mu_s>0$.
Risking a slight abuse of notation, we shall sometimes write  $\widetilde{P}=P\otimes \mu$.

Note that a random variable $X$ belongs to $L^0(\mathscr{O}^t)$ if, and only if, seen as a process $X=(X_s)$, it is $(\mathscr{F}_s)$-adapted up to time $t$ and constant afterwards.\footnote{
By ``constant afterwards'' we mean that $X_s=X_t$ for $s\geq t$.}
In particular, any $X \in L^0(\mathscr{O})$, seen as a process, is $(\mathscr{F}_s)$-adapted and it is clear that $L^0(\mathscr{O}^{t_1})\subseteq L^0(\mathscr{O}^{t_2})$ for any $t_1,t_2 \in \set{0,\dots,T}$ with $t_1\leq t_2$.

For any $X \in L^0(\mathscr{O})$, we denote by $\Delta X_s:=(X_s-X_{s-1}),$ with the convention $X_{-1}=0,$ so that $X_s=\sum_{k=0}^s \Delta X_s$.
\begin{remark}
	In what follows a process $X \in L^0(\mathscr{O})$ will be interpreted either as a discounted cumulative cash flow (discounted cumulative dividend) process, or as a discounted cash flow process (discounted dividend process).
	If $X$ is a discounted cumulative cash flow, then $\Delta X$ represents the discounted dividend process.
\end{remark}

From now through the end of this subsection we fix $t\in \set{0,1,\ldots,T}.$ For $q \in [1,+\infty]$, we denote by $\widetilde{\mathcal{M}}_{q,t}$, the set of probability measures $\widetilde{Q}$ on $\mathscr{O}$ absolutely continuous with respect to $\widetilde{P}$, such that $d\widetilde{Q}/d\widetilde{P} \in L^q(\mathscr{O})$ and $\widetilde{Q}=\widetilde{P}$ on $\mathscr{O}^t$.
In case $q=1$, and if no confusion arises, we will drop $q$ from the notations. Similarly, we denote by $\mathcal{M}_{t}$ the set of probability measures $Q$ on $\mathscr{F}_T$ absolutely continuous with respect to $P$, such that $dQ/dP \in L^1(\mathscr{F}_T)$ and $Q=P$ on $\mathscr{F}_t$.
Given $Q \in \mathcal{M}_{t}$ we denote by $\Gamma_t(Q)$ and $\mathcal{D}_t(Q)$ the set of optional random measures and predictable discounting processes\footnote{It is important to stress that process $D$ does not represent a financial discount factor. For the meaning and the role of this process we refer to Theorem~\ref{thm:robreb:proct}.} from time $t$ respectively, that is
\begin{align*}
	\Gamma_t(Q)&:=\Set{(\gamma_s)\in L^0_{+}(\mathscr{O}):\gamma_0=\ldots=\gamma_{t-1}=0\text{ and }\sum_{s=t}^T \gamma_s=1, \, Q\text{-almost surely}},\\
	\mathcal{D}_{t}(Q)&:=\Set{(D_t) \in L^0_{+}(\mathscr{O}):D_0=\ldots=D_t=1,\,Q\text{-almost surely, }D\text{ is predictable and decreasing}}.
\end{align*}

\begin{lemma}\label{lemma:repGammaQ}
	Let $Q \in \mathcal{M}_{t}$. There exists a one-to-one relation between $\gamma \in \Gamma_t(Q)$ and $D \in \mathcal{D}_t(Q)$ given by
	\begin{equation}
		D_0=1,\quad \text{and}\quad D_s=1-\sum_{k=0}^{s-1}\gamma_k,\quad \text{for}\quad0<s\leq T,
		\label{eq:onetoone:relation01}
	\end{equation}
	\begin{equation}
		\gamma_s=D_s-D_{s+1},\quad\text{for}\quad 0\leq s<T\quad\text{and}\quad\gamma_{T}=1-\sum_{k=0}^{T-1}\gamma_k=D_T.
		\label{eq:onetoone:relation02}
	\end{equation}
	Furthermore, for any $X \in L^0(\mathscr{O})$, it holds
	\begin{equation}
		\langle \gamma , X\rangle_t:=\sum_{s=t}^T \gamma_s X_s=X_t+\sum_{s=t+1}^T D_s \Delta X_s=:(D\bullet X)_t
		\label{eq:onetoone:relation03}
	\end{equation}
	with the convention that $D_{T+1}=0$.

	Finally, $\widetilde{Q}\in \widetilde{\mathcal{M}}_{t}$ if and only if there exists $Q \in \mathcal{M}_{t}$ and $\gamma \in \Gamma_t(Q)$ or the corresponding $D \in \mathcal{D}_{t}(Q)$ such that\footnote{Where $Q\otimes \gamma$ has to be understood as the product measure with density $(Z_t\frac{\gamma_t}{\mu_t})$, whereby $Z_t=dQ/dP_{\mid \mathscr{F}_t}$ and $Q\otimes D$ is the  product measure with density $(Z_t\frac{(D_t-D_{t+1})}{\mu_t})$.}
	$\widetilde{Q}=Q\otimes \gamma$ or $\widetilde{Q}=Q\otimes D$.
\end{lemma}
This was proven in \citep{AcciaioFollmerPenner2010}.
Note that the additional term $X_{t}$ in \eqref{eq:onetoone:relation03} of the integration by part is missing in \citep{AcciaioFollmerPenner2010}.
Next we define the sets\footnote{Analogously, we define the sets $\mathcal{M}\otimes_t \Gamma,$ and $\mathcal{M}\otimes_{q,t} \Gamma, q\in(1,\infty]$. }
\begin{align}
	\mathcal{M}\otimes_t \mathcal{D}&:=\Set{Q\otimes D:Q \in \mathcal{M}_{1}\text{ and }D \in \mathcal{D}_{t}(Q)};\\
\mathcal{M}\otimes_{q,t} \mathcal{D}&:=\Set{Q \otimes D:Q \in \mathcal{M}_{t}, D \in \mathcal{D}_{t}(Q),\text{ and }Q \otimes D \in \widetilde{\mathcal{M}}_{q,t} }, \quad q \in (1,+\infty].
\end{align}

\begin{remark}\label{rem:tildeq}
	By means of Lemma \ref{lemma:repGammaQ}, it holds $\widetilde{Q}\in \widetilde{\mathcal{M}}_{q,t}$ if and only if $\widetilde{Q}=Q \otimes D \in \mathcal{M}\otimes_{q,t} \mathcal{D}$, or $\widetilde{Q}=Q \otimes \gamma \in \mathcal{M}\otimes_{q,t} \Gamma$, $q\in[1,\infty]$.
\end{remark}

Following \citep{KupperVogelpoth2009}  we define the conditional $p$-norm
\begin{equation*}
	\norm{X}_{t,p}:=
	\begin{cases}
		\displaystyle E_{\widetilde{P}}\left[ \abs{X}^p \Mid \mathscr{O}^t\right]^{1/p},&\text{ if }p<\infty\\
		\\
		\displaystyle \essinf\Set{\xi \in L^0(\mathscr{O}^t):\abs{X}\leq \xi},&\text{ if }p=\infty,
	\end{cases}
\end{equation*}
on the basis of which we define the spaces
\begin{equation*}
	L^{t,p}\left( \mathscr{O} \right):=\Set{X \in L^0\left( \mathscr{O} \right):\norm{X}_{t,p}\in L^0\left( \mathscr{O}^t \right)}.
\end{equation*}
By means of \citep[Proposition 4.4]{KupperVogelpoth2009}, it holds that
\begin{equation}
	L^{t,p}\left( \mathscr{O} \right)=L^0\left( \mathscr{O}^t \right)  L^p\left( \mathscr{O} \right), \quad 1\leq p\leq \infty.
	\label{splitit}
\end{equation}
It is shown in \citep{KupperVogelpoth2009} that $(L^{t,p}( \mathscr{O}),\norm{\cdot}_{t,p})$, with the order of almost sure dominance, is an $L^0\left( \mathscr{O}^t \right)$--normed module lattice.
For a fixed $0\leq t\leq T$ and $1\leq p\leq \infty$, we let $\mathcal{X}=L^{t,p}(\mathscr{O})$.
We equip $\mathcal{X} = L^{t,p}(\mathscr{O})$ with the $\norm{\cdot}_{t,p}$-topology for $1\leq p<\infty$, or the conditional weak${}^\ast$-topology $\sigma(\mathcal{X},L^{t,1})$ if $p=\infty$.

We say that a functional $\alpha:\mathcal{X}\to \bar{L}^0(\mathscr{O}^t)$ is monotone  if $\alpha(X)\geq \alpha(Y)$ whenever $X \geq Y$ $\widetilde{P}$-almost surely\footnote{The monotonicity in this case coincides with the monotonicity with respect to the cone $\cK=\set{X\geq 0}$}.

\begin{theorem}\label{thm:robreb:proct}
	Let $\alpha:\mathcal{X}\to \bar L^{0}(\mathscr{O}^t)$ be an upper semicontinuous conditional assessment index.
	Then $\alpha$ has a robust representation of the form
	\begin{equation}\label{thm:eq:robrep:proct}
		\alpha\left( X \right)=\essinf_{\widetilde{Q} \in \widetilde{\mathcal{M}}_{q,t} } R\left( \widetilde{Q}, E_{\widetilde{Q}}\left[ X \mid \mathscr{O}^t \right] \right),
	\end{equation}
	for a unique minimal risk function $R:\widetilde{\mathcal{M}}_{q,t}\times \bar{L}^0(\mathscr{O}^t)\to \bar{L}^0(\mathscr{O}^t)$.

	This robust representation can be written in the following form
	\begin{align}
		\alpha_s\left(X\right)&=f_s(X_s), \quad s\leq t-1, \label{eq:daiProcS}
	\end{align}
	and,
	\begin{align}
		\alpha_{s}\left( X \right)=\alpha_t(X) & =\essinf_{Q\otimes \gamma \in \mathcal{M}\otimes_{q,t} \Gamma}R^{\prime}_{t}\left( Q\otimes \gamma, E_Q\left[ \sum_{k=t}^T \gamma_k X_k\Mid \mathscr{F}_t\right]\right) \label{eq:representation1} \\
		& =\essinf_{Q\otimes D \in \mathcal{M}\otimes_{q,t} \mathcal{D}}R^{\prime}_{t}\left( Q\otimes D, X_t+E_Q\left[ \sum_{k=t+1}^T D_k\Delta X_k\Mid \mathscr{F}_t\right]\right),\quad s\geq t ,\label{eq:representation2}
	\end{align}
	for an unique right-continuous increasing functions $f_s:L^0_s\to \bar{L}^0(\mathscr{F}_s)$ and minimal risk functions $R^{\prime}_t:\mathcal{M}\otimes_{q,t} \Gamma\times \bar{L}^0(\mathscr{F}_t)\to \bar{L}^0(\mathscr{F}_t)$.
\end{theorem}

\begin{remark}
	From the financial point of view, the representation \eqref{eq:representation1} is meaningful if $X$ is a discounted cash flow (discounted dividend process), and the representation \eqref{eq:representation2} is meaningful if $X$ is a discounted cumulative cash flow (discounted cumulative dividend process).
\end{remark}

\begin{proof}

	Since $\alpha$ is monotone with respect to cumulative cash flows, it holds $X\succcurlyeq Y$ if and only if $X-Y \in \mathcal{K}:=\set{U \in \mathcal{X}: U\geq 0}$ and so $\mathcal{K}^\circ=\set{Z \in L^{t,q}(\mathscr{O}):Z\geq 0}.$ We will make use of the normalized polar cone  $\mathcal{K}^\circ_{1}:=\set{Z \in L^{t,q}(\mathscr{O}):Z\geq 0\text{ and }E_{\widetilde{P}}[Z\mid \mathscr{O}^t]=1}$, which can be identified with $\widetilde{\mathcal{M}}_{q,t}$.
	Applying Theorem~\ref{thm:robrep} and Remark \ref{rem:normalizedset}, there exists a unique minimal conditional risk function $R:\widetilde{\mathcal{M}}_{q,t}\times \bar{L}^0(\mathscr{O}^t)\to \bar{L}^0(\mathscr{O}^t)$ such that  the representation \eqref{thm:eq:robrep:proct} holds true.

	To show the second claim of the theorem assume first that $p=\infty$.
	First note that
	\begin{equation}\label{eq:expCondO0}
		E_{\widetilde{Q}}\left[ X\Mid \mathscr{O}^t \right]=\left( X_0',\ldots,X_{t-1}',E_{Q}\left[ \langle \gamma, X \rangle_t\Mid \mathscr{F}_t \right],\ldots, E_{Q}\left[ \langle \gamma, X \rangle_t\Mid \mathscr{F}_t \right]\right),
	\end{equation}
	for all $X \in L^{t,\infty}(\mathscr{O})$ and all $\widetilde{Q}=Q\otimes \gamma$, where $Q \in \mathcal{M}_{t}$ and $\gamma \in \Gamma_t(Q)$, and where $X'$ is any element of $L(\mathscr{O})$. Indeed, suppose that $X \in L^{t,\infty}\left( \mathscr{O} \right)$, and denote by $Y$ the random variable on the right hand side of \eqref{eq:expCondO0}.
	Let $A=(A_0, A_1, \ldots, A_t, A_t, \ldots, A_t)$ such that $A_s \in \mathscr{F}_s$ for any $s\leq t$.
	Then, 	
	\begin{align*}
		E_{\tilde{Q}}\left[ X 1_A \right]
		& =\sum_{s=0}^{t-1} E_Q\left[X_s 1_{A_s} \gamma_s\right]+ \sum_{s=t}^T E_Q\left[ X_s 1_{A_t} \gamma_s\right]\\
		& = 0 + E_Q\left[ E_Q\left[\sum_{s=t}^T X_s\gamma_s\Mid \mathscr{F}_t\right] 1_{A_t} \right] \\
		& = E_{\tilde{Q}}\left[ Y 1_{A} \right],
	\end{align*}
	and hence \eqref{eq:expCondO0} is proved. For convenience, we will take $X'=X$ in what follows.

	By Remark~\ref{rem:tildeq}, $\widetilde{Q}\in \widetilde{\mathcal{M}}_{t}$ if and only if $\widetilde{Q}=Q\otimes \gamma$ for $Q \in \mathcal{M}_{t}$ and $\gamma \in \Gamma_t(Q)$.
	For $s\leq t-1$ we use locality for $A=\Omega \times \set{s} \in \mathscr{O}^t$ which yields $1_{\set{s}} \alpha(1_{\set{s}}X)=1_{\set{s}} \alpha(X)$ since $1_{A}=1_{\set{s}}$.
	Thus, $\alpha_s(X) = \alpha_s(0,\ldots,X_s,0,\ldots) =: \alpha_{s}(X_s)$.
	Since\footnote{By $1_{\set{s}}(Q\otimes \gamma)$ we naturally mean the density of $Q\otimes \gamma$ with respect to $\widetilde{P}$ at time $s$.} $1_{\set{s}}(Q\otimes \gamma)=1_{\set{s}}$, for any $Q\in\mathcal{M}_t, \gamma\in\Gamma_t(Q)$, and using locality of $R$ and \eqref{eq:expCondO0}, we deduce that
	\begin{equation*}
		\alpha_s(X)=\alpha_s(X_s)=\essinf_{Q\otimes \gamma \in \mathcal{M}\otimes_{t}\Gamma}R_s(1_{\set{s}}(Q\otimes \gamma),(0,\ldots,X_s,0,\ldots))=:f_s(X_s), \quad s\leq t-1,
	\end{equation*}
	and thus \eqref{eq:daiProcS} is established.
	In the case $s\geq t$ we apply locality to the set $\Omega\times\set{t,\ldots,T}$.
	Hence, we see that $\alpha_s(X)$ is equal to $\alpha_t(X)$ for all $s\geq t$ and using \eqref{eq:expCondO0} we get
	\begin{equation}\label{shortcut}
		\alpha_t(X)= \essinf_{Q\otimes \gamma \in \mathcal{M}\otimes_t \Gamma} R^{\prime}_t\left(Q\otimes \gamma,E_{Q}\left[ \langle \gamma, X \rangle_t\Mid \mathscr{F}_t \right]\right),
	\end{equation}
	where
	\begin{equation*}
		R^{\prime}_t\left(Q\otimes \gamma,s_t\right) :=R_t \left(Q\otimes \gamma, \left(0,\ldots, 0, s_t,\ldots,s_t\right)\right), \quad s_t \in \bar{L}^0(\mathscr{F}_t),
	\end{equation*}
	is a uniquely determined risk function. This proves the representation \eqref{eq:representation1}. By Lemma~\ref{lemma:repGammaQ} and \eqref{eq:representation1}, the represetnation \eqref{eq:representation2} follows immediately.

	As for the case $1\leq p< +\infty$,	in view of Remark \ref{rem:tildeq}, and proceeding analogously as above, we conclude that \eqref{eq:representation1} and \eqref{eq:representation2} are satisfied.
\end{proof}

\begin{remark}\label{rem:imp}
	It is in place here to remark that the assessment index $\alpha$ considered in this subsection corresponded to the fixed $t$. It would be then appropriate to denote it as, say, $\alpha^t=(\alpha^t_0,\ldots,\alpha^t_T).$ We would then refer to the collection $\set{\alpha^t,\ t=0,1,\ldots,T}$ as to \emph{dynamic assessment index}.
\end{remark}

\subsection{Path Dependent Dynamic Assessment Indices}\label{sec:PathDep}
Throughout this section we interpret $X$ as the discounted cumulative cash-flow.

It is seen from representation \eqref{eq:representation1} that $\alpha^t_t$ (cf. Remark~\ref{rem:imp}) only assesses the future of the process $X$, that is it only assesses $X_t,\ldots,X_T$, while $\alpha_s^t, \ s<t,$ is just a function of $X_s$.
This is a drawback since the past evolution of $X$ is not taken into account when assessing $X$ at time $t$ via $\alpha^t_t$, which for some applications may be an unwanted feature.

In this section  we propose an alternative approach, which assess $X$ at time $t$ accounting for the path evolution of $X$ time $t$.
Given $0\leq s\leq \tilde{s}\leq T$, we denote by $1_{[s,\tilde{s}]}$ a process, such that $1_{[s,\tilde{s}]}(u)=1$ for $s\leq u \leq \tilde{s}$, and $1_{[s,\tilde{s}]}(u)=0$ otherwise. Accordingly, we use the notation  $X_{[s,\tilde{s}]}$ for the random vector $X1_{[s,\tilde{s}]}=(0,\dots,0,X_s,\dots,X_{\tilde{s}},0,\dots,0)$.
Process $X$ stopped at time $t$ is written as $X^t$, that is $X^t=X_{\cdot\wedge t}$.
We recall the definition of the space $L^0(\mathscr{O}^t)$ (cf. \eqref{eq:ot}), and we define
\begin{align*}
	L^0(\mathscr{O}_{[s,\tilde{s}]})&:=\Set{X_{[s,\tilde{s}]}: X \in L^0(\mathscr{O})}.
\end{align*}
We remark that $\mathscr{O}_{[s,\tilde{s}]}$ is understood as the optional $\sigma$-algebra generated by processes $X_{[s,\tilde{s}]}$.
Hence, for a fixed  $t$ we may decompose any process $X \in L^0(\mathscr{O})$ as follows
\begin{equation}\label{eq:decompint}
	X=X_{[0,t-1]}+X_{[t,T]}=X^{t-1}+(X_{[t,T]}-X_{t-1}1_{[t,T]}),
\end{equation}
where $X_{[0,t-1]}\in L^0(\mathscr{O}_{[0,t-1]})$, $X_{[t,T]}\in L^0(\mathscr{O}_{[t,T]})$ and $X^{t-1} \in L^0(\mathscr{O}^{t-1})$.

It is evident that $L^0(\mathscr{O}_{[t,T]})$ is an $L^0(\mathscr{F}_{t})$-module\footnote{For the multiplication $\lambda X_{[t,T]}=(0,\ldots,0,\lambda X_{t},\ldots,\lambda X_T)$, $\lambda \in L^0(\mathscr{F}_t)$.}.
We further define
\begin{align*}
	\hat{\mathcal{M}}_{q,t}&:=\Set{\hat{Q}:\hat{Q}\text{ measure on }\Omega \times \set{0,\ldots,T},\hat{Q}\prec
	\hat{P}:=P\otimes \mu, d\hat{Q} / d \hat{P} \in L^q(\mathscr{O}_{[t,T]})},
\end{align*}
where $\mu$ is a measure on $\set{t,\ldots,T}$ such that $\mu_s>0$ for every $s \in \set{t,\ldots,T}$.
We further denote
\begin{align*}
	\hat{\mathcal{M}}\otimes_{q,t}\hat{\mathcal{D}}&:=\Set{Q\otimes D: Q\in \mathcal{M}_t, D \in \mathcal{D}_{t}(Q),\text{ and }Q\otimes D \in \hat{\mathcal{M}}_{q,t}}.
\end{align*}

\begin{remark}
	In this setting, let $\hat{Q}\in \hat{\mathcal{M}}_{1,t}$, and denote by $\Lambda =d\hat{Q}/d\hat{P} \in L^1(\mathscr{O}_{[t,T]})$.
	It holds that $U=(U_s)_{s=t}^T$, where $U_s=E_P[\sum_{k=s}^T \Lambda_k \mu_k\mid \mathscr{F}_s]$ for $s \in \set{t,\ldots,T},$
	is a super martingale fulfilling additionally $E_P[U_{t+1}\mid \mathscr{F}_t]=U_t=1$.
	Hence, using the It\^o-Watanbe decomposition $U=ZD$ where $D$ is a predictable decreasing process and $Z$ is a martingale, it follows that $D_t=1$ and $Z_T$ is a density of a probability measure $Q \in \mathcal{M}_t$.
	Reciprocally, $\Lambda_k=Z_k(D_k-D_{k+1})/\mu_k$ for every $k=t,\ldots,T-1$, and $\Lambda_T=Z_TD_T/\mu_T,$ where $Z$ is a martingale and $D$ is
	a predictable decreasing process with $D_{t}=1$, defines a density process for some $\hat{Q}\in \hat{\mathcal{M}}_{1,t}$.
	Hence, for every $X_{[t,T]} \in L^{t,p}(\mathscr{O}_{[t,T]})$,\footnote{In analogy to $L^{t,p}(\mathscr{O})=L^0(\mathscr{O}^t)L^p(\mathscr{O})$,
	we have  {$L^{t,p}(\mathscr{O}_{[t,T]})=L^0(\mathscr{F}_t)L^p(\mathscr{O}_{[t,T]})$.}}  it follows that
	\begin{align}\label{expectation}
		E_{\hat{Q}}\left[ X_{[t,T]}\Mid \mathscr{F}_t \right]
		& = E_{\hat{P}}\left[\Lambda X_{[t,T]}\Mid \mathscr{F}_t \right] = E_P\left[ \sum_{k=t}^T \Lambda_k X_k \mu_k \Mid \mathscr{F}_t \right] \nonumber \\
		&=E_{Q}\left[\sum_{k=t}^{T-1}(D_k-D_{k+1})X_k+D_TX_T \Mid \mathscr{F}_t \right] = E_{Q}\left[ X_{t}+ \sum_{k=t+1}^TD_k\Delta X_k \Mid \mathscr{F}_t \right].
	\end{align}
\end{remark}

We finally set
\begin{equation}\label{eq:xtp}
	\mathcal{X}^p_t:=\set{X\in L^0(\mathscr{O}):X_{[t,T]}\in L^{t,p}(\mathscr{O}_{[t,T]})}.
\end{equation}
Note that $\cX_{t}^p\subset\cX_{t+1}^p$.
\begin{definition}
	A function $\alpha:\mathcal{X}^{p}_t\to \bar{L}^0(\mathscr{F}_t)$ is called an upper semicontinuous
	\emph{path dependent assessment index} if for every fixed path $\bar{X}\in L^0(\mathscr{O}^{t-1})$, the function
	\begin{equation}
		X_{[t,T]}\longmapsto \alpha\left( \bar{X}_{[0,t-1]} + X_{[t,T]} \right),\quad X_{[t,T]} \in L^{t,p}(\mathscr{O}_{[t,T]}),
		\label{}
	\end{equation}
	is an upper semicontinuous assessment index.
\end{definition}
\begin{theorem}\label{th:repPast1}
	Let $\alpha$ be an upper semicontinuous path dependent assessment index.
	Then it has a robust representation of the form
	\begin{equation}
		\alpha(X)=\essinf_{Q\otimes D \in \hat{\mathcal{M}}\otimes_{q,t}\hat{\mathcal{D}} } R\left(X_{[0,t-1]};Q\otimes D;
		E_Q\left[  X_{t}+ \sum_{k=t+1}^TD_k\Delta X_k\Mid \mathscr{F}_t\right] \right),
		\label{thm:eq:repPast1}
	\end{equation}
	for a unique function $R:L^0(\mathscr{O}_{[0,t-1]})\times  \hat{\mathcal{M}}\otimes_{q,t}\hat{\mathcal{D}} \times
	\bar{L}^{0}(\mathscr{F}_t) \to \bar{L}^0(\mathscr{F}_t)$ for which $R(X_{[0,t-1]},\cdot,\cdot):
	\hat{\mathcal{M}}\otimes_{q,t}\hat{\mathcal{D}}\times \bar{L}^{0}(\mathscr{F}_t) \to \bar{L}^0(\mathscr{F}_t)$
	is a maximal risk function for every $X_{[0,t-1]}\in L^0(\mathscr{O}_{[0,t-1]})$.
\end{theorem}
\begin{proof}
	First, we fix  $\bar{X}\in L^0(\mathscr{O}^{t-1})$ and we apply Theorem~\ref{thm:robrep} to
	$\alpha(\bar{X}+\cdot)$ in the fashion analogous the the proof of  Theorem \ref{thm:robreb:proct}
	in order to get the following representation

	\begin{equation}\label{thm:eq:robrep:proct:path}
		\alpha\left( \bar{X}_{[0,t-1]}+ X_{[t,T]}\right)=\essinf_{\hat{Q}\in \hat{\mathcal{M}}_{q,t}} \bar
		R\left(\bar{X}_{[0,t-1]}, \hat{Q}, E_{\hat{Q}}\left[ X_{[t,T]} \mid \mathscr{F}_t \right] \right).
	\end{equation}
	Similarly as in Remark~\ref{rem:tildeq}, we also have that $\hat{Q}\in \hat{\mathcal{M}}_{q,t}$ if and only
	if $\hat{Q}=Q \otimes D \in \hat{\mathcal{M}}\otimes_{q,t}\hat{\mathcal{D}}$.
	Hence, using \eqref{expectation} in representation \eqref{thm:eq:robrep:proct:path}, we conclude the proof.

\end{proof}

Note that $\alpha$ is no longer local with respect $\mathscr{O}^{t-1}$ on $\Omega \times \set{0,\ldots,T}$.
Let us now consider the following illustrating example.
\begin{example}\label{exep:pathdep}
	Let us consider a function $\alpha:\mathcal{X}_t^p\to \bar{L}^0(\mathscr{F}_t)$ given by the following formula
	\begin{equation}
		\alpha(X)=\sum_{k=0}^{t-1}D_k'\Delta X_k+\essinf_{Q\otimes D\in \hat{\mathcal{M}}\otimes_{q,t}\hat{\mathcal{D}}}R'\left(Q\otimes D,X_{t}+E_{Q}\left[\sum_{k=t+1}^TD_k\Delta X_k \Mid \mathcal{F}_t\right]\right)
		\label{twicediscount}
	\end{equation}
	where $D'$ is an adapted process, and  $R'(\cdot,\cdot):\hat{\mathcal{M}}\otimes_{q,t}\hat{\mathcal{D}}\times \bar{L}^{0}(\mathscr{F}_t) \to \bar{L}^0(\mathscr{F}_t)$ is a maximal risk function.
	Then, such $\alpha$ is a an upper semicontinuous path dependent assessment index.

	The whole process $(D'_0,\ldots,D'_{t-1}, 1,D_{t+1},\ldots,D_T)$ may be interpreted as weighing the past and the future of the cash flows, relative to the present time $t$.

	Depending on the specification of $D'_k$, we get,
	\begin{itemize}
		\item if all $D'_k=0$, a representation of path independent assessment indices.
		\item If all $D'_k=1$, then $\sum_{k=0}^{t-1} \Delta X_k=X_{t-1}$, which means that $\alpha$ depends only on the assessment of the future returns starting at the previous level of wealth $X_{t-1}$.
		\item Changing the parameter $D'_k$ in between, one puts more or less weight on the past evolution of returns.
	\end{itemize}
This kind of past dependence indicates how the past evolution of discounted cumulative cashflows may influence the present assessment of the entire investment process.
On the one hand such an index provide a model that explains optimistic/pessimistic assessment due to recent period of good/bad performances.
One the other hand, such an index may provide some guidelines to the regulator to implement contra-cyclical regulations.
Indeed, they could require $D'$ to be dependent on the past returns, penalising more in the presence of recent overperformance whereas being less demanding in period of recent drawdown.
Such a weighting factor reflecting this feature could take the form
	\begin{equation*}
		D'_k=\exp\left( 0.08-\frac{\Delta X_k}{X_k} \right),
	\end{equation*}
	where $8\%$ were a reasonable annual return for a banking institution.
\end{example}

\begin{remark}\label{rem:imp1}
		Similarly as in Remark~\ref{rem:imp} we observe that the assessment index $\alpha$ considered in this subsection corresponded to the fixed $t$. It would be then appropriate to denote it as, say, $\alpha_t.$ We would then refer to the collection $\set{\alpha_t,\ t=0,1,\ldots,T}$ as to \emph{dynamic path dependent assessment index}.
\end{remark}

\section{Dynamically Consistent Assessment Indices }\label{sec04}
In this section we discuss the key notion of dynamic consistency with regard to assessment indices.
Here, we only focus on the so called strong dynamic consistency for path dependent assessment indices.
For other notions of time consistency we refer to e.g. \citet{AcciaioFollmerPenner2010}, \citet{AcciaioPenner2010} and references therein,  with regard to dynamic risk measures, and we refer to \citet{BCZ2010} and \citet{BiaginiBion-Nadal2012} with regard to acceptability indices.

We consider a dynamic path dependent assessment index $\alpha=\set{\alpha_t, \ t=0,\ldots, T}$ (cf. Remark~\ref{rem:imp1}).

\begin{definition}\label{def:strocons}
	We say that $\alpha$ is \emph{strongly time consistent} if for any $X,Y\in  \mathcal{X}^p_t$ and $t$ such that  $X_{[0,t]}=Y_{[0,t]}$ the following implication is true\footnote{{Recall that $\cX_{t}^p\subset\cX_{t+1}^p$. }}
	\begin{equation*}
		\alpha_{t+1}(X)\geq \alpha_{t+1}(Y) \quad \text{implies}\quad \alpha_t(X)\geq \alpha_t(Y).
	\end{equation*}
\end{definition}

\begin{remark}
One needs to observe that the notion of strong time consistency seems to be inappropriate in the case of scale invariant assessment indices.
Indeed, let $\alpha$ be scale invariant and strongly time consistent. Assume that $X_{[0,t]} Y_{[0,t]}\geq 0$ and $\alpha_{t+1}(X)\geq \alpha_{t+1}(Y)$.
Then, there exists $\lambda\in L^0_{++}(\mathscr{O}^t)$ such that $\lambda X_{[0,t]} = Y_{[0,t]}$, and in view of scale invariance of $\alpha$, we have that $\alpha_t(X)\geq \alpha_t(Y)$.
Thus the condition $X_{[0,t]}=Y_{[0,t]}$ appears to be irrelevant for the strong time consistency in this case, which is unreasonable  from the risk management point   of view.
Consequently, a different notion of time consistency is needed in case of scale invariant assessment indices. One such possible notion was introduced and studied in \cite{BCZ2010}.

Moreover, as shown below, the strong time consistency is strongly related to existence of a certainty equivalent, which fails to exists (see Remark~\ref{rem:CerEqNotGeneral}) for scale invariant assessment indices.
\end{remark}

In order to derive a version of the so called Bellman principle, some additional assumptions have to be done.
We suppose throughout this section that $X_{[t,T]}\mapsto \alpha_t(X_{[0,t-1]}+X_{[t,T]})$ fulfills the assumptions of Proposition \ref{prop:relation:alpha:C:bis} with the boundedness assumption given for $m_1,m_2 \in L^p(\mathscr{F}_t)$ rather than $L^0(\mathscr{F}_t)$.

Let us define a family of functionals $C_t: \mathcal{X}^p_t\to \bar{L}^0(\mathscr{F}_t)$ for $t=0,1,\ldots,T,$ by
\begin{equation*}
	C_{t}(X):=\essinf \Set{m_t \in L^p(\mathscr{F}_t):\alpha_{t}(X_{[0,t-1]}+m_{t}1_{[t,T]})\geq \alpha_{t}(X)}.
\end{equation*}
According to Proposition \ref{prop:relation:alpha:C:bis}, for each $t,$   $X_{[t,T]}\mapsto C_{t}(X_{[0,t-1]}+X_{[t,T]})$ is an upper semicontinuous (path dependent) assessment index taking values into $L^p(\mathscr{F}_t)$ such that
\begin{equation*}
	\alpha_t(X)\geq \alpha_t(Y)\quad \text{if, and only if}\quad C_t(X)\geq C_t(Y).
\end{equation*}
In particular $C_{t}(X_{[0,t-1]}+C_{t}(X)1_{[t,T]})=C_t(X)$.
In addition, the family $\alpha$ is strongly time consistent
if and only if the family  $C:=(C_t)$ is strongly time consistent.

With this at hand, we may formulate the following version of the celebrated Bellman principle.
\begin{proposition}\label{prop:bellman}
	Under the assumptions adopted in this section, if $\alpha$ is strongly time consistent, the corresponding family $C$ of certainty equivalents satisfies, for each $t=0,\ldots,T-1$,
	\begin{equation}\label{eq:strongtimeconst}
		C_{t}\left( X \right)=C_{t}(X_{[0,t]}+C_{t+1}(X)1_{[t+1,T]}),\quad X \in \mathcal{X}^p_t.
	\end{equation}
\end{proposition}
\begin{proof}
	Since $C_{t+1}$ is a certainty equivalent, it follows that $C_{t+1}(X)=C_{t+1}(X_{[0,t]}+C_{t+1}(X)1_{[t+1,T]})$.
	By means of the boundedness assumption, $C_{t+1}(X)\in L^p(\mathscr{F}_t)$, and so defining $Y=X_{[0,t]}+C_{t+1}(X)1_{[t+1,T]}$, it follows that $Y \in \mathcal{X}_t^p$ and $Y_{[0,t]}=X_{[0,t]}$.
	Thus, the strong time consistency applied to $C$ yields \eqref{eq:strongtimeconst}.
\end{proof}
From now on, we consider certainty equivalent corresponding to assessment indices fulfilling the conditions from Proposition \ref{prop:bellman}.
Note that for $X_{[0,t]}\in L^0(\mathscr{O}^t)$, the function $C_t:L^p_{t+1}(\mathscr{F}_t)\to L^p(\mathscr{F}_t)$, $Y \mapsto C_t(X_{[0,t]}+Y1_{[t+1,T]})$ is an upper semicontinuous assessment index, and we denote by $R_{t,t+1}$ its corresponding minimal risk function, for which it holds
\begin{equation}
	C_{t}(X_{[0,t]}+Y1_{[t+1,T]})=\essinf_{Q\otimes D \in \mathcal{MD}_t^{t+1}}R_{t,t+1}\left( X_{[0,t-1]}, Q\otimes D, X_t+E_{Q}\left[ D\left( Y-X_{t} \right)\Mid \mathscr{F}_t \right] \right),
	\label{}
\end{equation}
where
\begin{equation*}
	\mathcal{MD}_t^{t+1}:=\Set{ Q\otimes D: Q \in \mathcal{M}_t^{t+1}, 0\leq D\leq 1 \text{ and }D\text{ is }\mathscr{F}_t\text{-measurable}},
\end{equation*}
whereby $\mathcal{M}_t^{t+1}$ denotes the set of probability measures $Q$ on $\mathscr{F}_{t+1}$ such that $Q\prec P$ and $Q=P$ on $\mathscr{F}_t$.
As a convention, we set $\mathcal{MD}_T^{T+1}=\set{1}$ since $C_{T}(X)=X_T$.
\begin{theorem}
	If $\alpha=(\alpha_t)$  is a strongly time consistent sequence  of  path dependent assessment indices fulfilling the assumptions of Proposition \ref{prop:relation:alpha:C:bis}, then
	\begin{equation}
		C_{t}\left( X \right)=\essinf_{Q\otimes D \in \mathcal{MD}_t^{t+1} } F_{t}\left( Q\otimes D,X\right);\quad X \in \mathcal{X}_t^p,
		\label{}
	\end{equation}
	where
	\begin{multline}
		F_{t}\left( Q\otimes D,X \right)=\\
		\essinf_{\bar{Q}\otimes \bar{D}\in \mathcal{MD}_{t+1}^{t+2}}R_{t,t+1}\left( X_{[0,t-1]}, Q\otimes D, E_Q\left[ D\Big( F_{t+1}(\bar{Q}\otimes \bar{D},X)-X_t\Big)+ X_t \Mid \mathscr{F}_t\right] \right),
		\label{}
	\end{multline}
	for $t\leq T-1$ and
	\begin{equation}
		F_{T}(Q\otimes D,X)=X_T, \quad Q\otimes D \in \mathcal{MD}_T^{T+1}=\set{1}.
		\label{}
	\end{equation}
\end{theorem}
\begin{proof}
	Let us prove the theorem for $t=T-1,T-2$; the rest of the proof follows by backward recursion.
	Clearly, $C_{T}(X)=X_T$.
	As for $t=T-1$, since $\mathcal{MD}_{T-1}=\mathcal{MD}_{T-1}^T$ and $R_{T-1}=R_{T-1,T}$, it holds
	\begin{align}
		C_{T-1}\left( X \right)&=\essinf_{Q\otimes D \in \mathcal{MD}_{T-1}^T}R_{T-1,T}\left( X_{[0,T-2]},Q\otimes D, E_{Q}\left[ D\Big( X_T-X_{T-1}\Big)+X_{T-1}\Mid \mathscr{F}_{T-1} \right] \right)\nonumber\\
		&=\essinf_{Q\otimes D \in \mathcal{MD}_{T-1}^T}F_{T-1}\left( Q\otimes D,X \right), \label{eq:bliblobly}
	\end{align}
	where
	\begin{multline*}
		F_{T-1}(Q\otimes D)=\\
		\essinf_{\bar{Q}\otimes \bar{D}\in \mathcal{MD}_{T}^{t+1}}R_{T-1,T}\left( X_{[0,T-1]},Q\otimes D,
		E_{Q}\left[ D\Big( F_{T}(\bar{Q}\otimes \bar{D},X)-X_{T-1}\Big)+X_{T-1}\Mid \mathscr{F}_{T-1} \right] \right),
	\end{multline*}
	since $F_{T}(\bar{Q}\otimes \bar{D},X)=X_T$ for all $\bar{Q}\otimes \bar{D}\in \mathcal{MD}_{T}^{T+1}$.
	
	For $t=T-2$, by time consistency, and since $C_{T-1}(X)$ is $\mathscr{F}_{T-1}$-measurable, we deduce that
	\begin{multline*}
		C_{T-2}\left( X \right)=C_{T-2}\left( X_{[0,T-2]}+C_{T-1}\left( X \right) 1_{[T-1,T]} \right)\\
		=\essinf_{Q\otimes D \in \mathcal{MD}_{T-2}^{T-1}}R_{T-2,T-1}\left( X_{[0,T-3]},Q\otimes D,E_{Q}\left[ D\Big(C_{T-1}(X)-X_{T-2}  \Big)+X_{T-2}\Mid \mathscr{F}_{T-2} \right] \right).
	\end{multline*}
	Since $s \mapsto R_{T-2,T-1}\left( X_{[0,T-3]},Q\otimes D,s \right)$ is right-continuous, by means of \eqref{eq:bliblobly} it follows that
	\begin{align*}
		R&_{T-2,T-1}\left( X_{[0,T-3]},Q\otimes D,E_{Q}\left[ D\Big(C_{T-1}(X)-X_{T-2}  \Big)+X_{T-2}\Mid \mathscr{F}_{T-2} \right] \right)\\
		&=\essinf_{\bar{Q}\otimes \bar{D} \in \mathcal{MD}_{T-1}^T} R_{T-2,T-1}\left( X_{[0,T-3]},Q\otimes D,E_{Q}\left[ D\Big(F_{T-1}(\bar{Q}\otimes \bar{D}, X)-X_{T-2}  \Big)+X_{T-2}\Mid \mathscr{F}_{T-2} \right] \right)\\
		&= F_{T-2}(Q\otimes D,X)
	\end{align*}
	which ends the proof.
\end{proof}
\begin{remark}
	Suppose that $\alpha$ is given by
	\begin{equation*}
		\alpha_t(X)=\sum_{k=1}^{T-1}D^\prime_k \Delta X_k+\beta_{t}(X_{[t,T]})
	\end{equation*}
	as in Example \ref{exep:pathdep}, where $D^\prime=(D_0,\ldots,D_{T-1}^\prime)$ is fixed, and $\beta$ is a strongly time consistent path independent assessment index.
	Then, it follows easily that $\alpha$ itself is a strongly time consistent AI.
\end{remark}

\section{Examples}\label{sec:Examples}

\subsection{Dynamic Gain-to-Loss Ratio}\label{sec:dglr}

We shall discuss here an important example of an assessment index, namely the dynamic Gain-to-Loss Ratio (dGLR).
This index, in fact, provides an example of a dynamic acceptability index, since it is scale invariant. It was introduced in \cite{BCZ2010}, in a slightly different form.
The version of dGLR given in Definition~\ref{def:dglrDiv} below is not strongly time-consistent in the sense of Definition~\eqref{def:strocons}, but it is time-consistent in the sense of \cite{BCZ2010}.

The prototype for the definition below is the classical measure of performance Gain-to-Loss Ratio (GLR):
given an integrable, real-valued random variable $X$,
GLR is defined as $\mathrm{GLR}(X):=\mathbb{E}(X)/\mathbb{E}(X^-)$, if $\mathbb{E}[X]>0$, {$\mathrm{GLR}(0)=+\infty$} and zero otherwise, where $X^-:=\max\{-X,0\}$.

In the rest of this Section we use the set-up of Section \ref{sec03}. In particular,  we fix a $t\in\set{1,\ldots,T}$, we take $\cX=L^{t,p}(\mathscr{O})$, and  we consider  $\cX$ to be an {$L^{0}(\mathscr{O}^t)$-module}. Recall that the cone $\cK$ in this case is given by $\cK=\set{X\in\cX : X\geq 0}$.

Here, any element $X\in \cX$ is considered to be a discounted dividend process, and the next definition gives a relevant formula for dGLR.

\begin{definition}\label{def:dglrDiv}
Here $X$ represents the discounted dividend process.  We define dGLR as follows
\begin{align}\label{eq:dGLRDiv}
dGLR_s(X)
& =
\begin{cases}
{G}(X), & s\geq t \\
  +\infty, & s \leq  t-1,
\end{cases}
\end{align}
where

\begin{align*}
{G}(X) & =
\begin{cases}
  \frac{E[  \sum_{s=t}^T X_s \mid \mathscr{F}_t]}
{E[\left(  \sum_{s=t}^T X_s \right)^- \mid \mathscr{F}_t]}, & \textrm{ on } B^X_1
 \\
{+\infty}, & \textrm{ on } B_2^X  \\
0, & \textrm{ on } B_3^X.
\end{cases}
\end{align*}
with $B^X_1:= \set{E[   \sum_{s=t}^T X_s \mid \mathscr{F}_t] >0}, \ B^X_2:=\esssup\set{A\in\mathscr{F}_t : 1_A\sum_{s=t}^T X_s=0}$, and $B_3^X:=(B_1^X\cup B_2^X)^c$.
\end{definition}
\noindent Note that for any $X\in\cX$, we have that $P(B_1^X \cap B_2^X)=0$, hence $G$ is well defined.

We will show that the above dGLR is monotone, quasi-concave, local, scale invariant and upper-semicontinuous.
Clearly it is enough to show that the properties are satisfied for the function $G$.
In the rest of the section we will use the notation $\widetilde{X} := \sum_{s=t}^T X_s$, for $X\in\cX$.

\noindent \underline{Monotonicity:}
Let $X,Y\in \cX$ be such that $X-Y\in \cK.$ Thus, $\widetilde{X}\geq \widetilde Y.$
We will need to consider all the following cases $\omega\in B_i^X\cap B_j^Y, \ i,j=1,2,3.$
First, note that for $\omega B_i^X\cap B_j^Y,$ with $i=1, j=3; i=j=2; i=2, j=3; i=j=3$, the inequality $G(X)(\omega)\geq G(Y)(\omega)$ is obviously satisfied.

Next, we consider the case $i=j=1$. Note that $B_1^X\cap B_1^Y=B_1^Y$. Consequently, on the set $B_1^Y$, we have $E[\widetilde{X} | \mathscr{F}_t] \geq E[\widetilde{Y} | \mathscr{F}_t] >0$, which immediately implies that $G(X)= E[\widetilde{X} | \mathscr{F}_t] / E[\widetilde{X}^- | \mathscr{F}_t] \geq E[\widetilde{Y} | \mathscr{F}_t]/ E[\widetilde{Y}^- | \mathscr{F}_t] = G(Y)$.

Since $1_{B_1^X\cap B_2^X}\widetilde{X}^-=0$,  we get that $E[\widetilde{X}^-|\mathscr{F}_t] = 0$ on $B_1^X\cap B_2^Y$, that consequently implies that $G(X)=+\infty=G(Y)$ on $B_1^X\cap B_2^Y$.

Also note that $P[B_3^X \cap B_2^Y]=0$. Indeed, for any $C\subset B^Y_2\cap \mathscr{F}_t$, we have that $E[1_C \widetilde{X} | \mathscr{F}_t]\geq0$. If, in addition,  $C\subset B_3^X$, then $E[1_C\widetilde{X} | \mathscr{F}_t] \leq 0$, and thus $1_C\widetilde{X} \equiv 0$, which implies that $C\subset B_2^X$. Since $B_2^X\cap B_3^X =\empty$, we have that $P[C]=0$.
Similarly, one can show that $P[B_i^X\cap B_j^Y] =0$ for $i=2, j=1; i=3,j=1$.
This proves the monotonicity.

\noindent \underline{Quasi-concavity:}
 Let  $X,Y\in \cX$, $\lambda\in L^0(\mathscr{O}^t)$  and $\lambda\in[0,1]$.
 It is enough to show that for any $x\in\bar{L}^0(\mathscr{F}_t)$ such that $G(X)\geq x$ and $G(Y)\geq x$, we have that $G(\lambda X + (1-\lambda)Y)\geq x$.
First, we consider the case $\omega\in B_1^X\cap B_1^Y$.
Then, on $B_1^X\cap B_1^Y$, we have that
\begin{align*}
E[ \widetilde X \mid \mathscr{F}_t ] \geq x E[\widetilde X^- \mid \mathscr{F}_t]\\
E[ \widetilde Y \mid \mathscr{F}_t ] \geq x E[\widetilde Y^- \mid \mathscr{F}_t].
\end{align*}

From here, since, $\lambda_t=\lambda_{t+1}=\ldots=\lambda_T$, and by convexity of $x\to x^-$,  we get
\begin{align*}
xE[(\lambda_t \widetilde{X} + (1-\lambda_t \widetilde{Y}))^- \ | \ \mathscr{F}_t]
& \leq x E[\lambda_t \widetilde{X}^-  + (1-\lambda_t) \widetilde{Y}^-\ | \ \mathscr{F}_t] \\
& \leq    x \lambda_t E[ \widetilde{X}^-  \ | \ \mathscr{F}_t] + x(1-\lambda_t) E[\widetilde Y^- \ | \ \mathscr{F}_t] \\
& \leq   \lambda_t E[ \widetilde X  \ | \ \mathscr{F}_t] + (1-\lambda_t) E[\widetilde Y \ | \ \mathscr{F}_t] \\
& =  E[ \lambda_t \widetilde X  + (1-\lambda_t)\widetilde Y \ | \ \mathscr{F}_t].
\end{align*}
From here, and since $B_1^X\cap B_1^Y \subset B_1^{\lambda X + (1-\lambda)Y}$, we conclude that $G(\lambda X+(1-\lambda)Y) \geq x$.

If $\omega \notin B_1^X\cap B_1^Y$,  then $x(\omega)\leq0$ or $x(\omega)=+\infty$, and hence clearly $G(\lambda X+(1-\lambda)Y)(\omega) \geq x(\omega)$.
Thus, the quasi-concavity of dGLR follows.

\noindent \underline{Locality:}
It is enough to prove that
\begin{equation}\label{eq:dGLRScaleInv}
1_AG(X)=1_AG(1_AX),
\end{equation}
for any $A\in\mathscr{F}_t$ and $X\in\cX$ such that $X=(0,\ldots,0,X_t,\ldots,X_T)$.

Clearly, the equality \eqref{eq:dGLRScaleInv} is satisfied for $\omega\in A^c$.
Since $A\cap B_1^X\ = A\cap B_1^{1_AX}$, by locality of the conditional expectation we conclude that \eqref{eq:dGLRScaleInv} holds true on $A\cap B_1^X$.
Also note that $A\cap B_2^X = A\cap B_2^{1_AX}$, and hence \eqref{eq:dGLRScaleInv} is satisfied on $A\cap B_2^X$. Moreover, the above imply that $A\cap B_3^X = A\cap B_3^{1_AX}$, which consequently shows that \eqref{eq:dGLRScaleInv} holds true on $A\cap B_3^X$. Thus, locality is proved.

\noindent \underline{Scale invariance:}
Note that for any $\lambda\in L^0_{++}(\mathscr{O}^t), X\in \cX^t$,  we have $\lambda X^- = (\lambda X)^-$, $B_i^X=B_i^{\lambda X}$ for $i=1,2,3$, and hence the scale invariance follows immediately.

\noindent \underline{Upper semicontinuity:}
We will show that the upper level sets $\cA^m = \set{X\in \cX : G(X)\geq m}$ are closed for any $m\in L^0(\mathscr{F}_t)$.
If $m\leq 0$, then $\cA^M=\cX$ which is obviously closed. Next, assume that $m >0$.
Then, note that
$$
\cA^m= \Set{ X\in\cX : 1_{B_1^{X}} \frac{E[\widetilde X |\mathscr{F}_t]}{E[\widetilde X ^- |\mathscr{F}_t]}  + 1_{B_2^{X}}\infty \geq m }.
$$
Next, consider the set
$$
\cB^m =\Set{X\in\cX : E[\widetilde X |\mathscr{F}_t] - m E[\widetilde X ^- |\mathscr{F}_t] \geq 0}.
$$
Observe that if $X\in\cB^m$, then $P(B^X_3)=0$. Consequently, we have that $\cA^m=\cB^m$.
The closedness of $\cA^m$ follows from the above equality and from continuity of the function $h(X) = E[\widetilde X |\mathscr{F}_t] - m E[\widetilde X ^- |\mathscr{F}_t] $.

Finally, the case of general $m$ is treated by locality.

\bigskip

\noindent Next, we will provide a robust representation for GLR, using the results from Section~\ref{sec:scaleInv}.
\begin{proposition}
 The unique minimal risk function $R$ in representation \eqref{eq:robrep} of GLR has the following form
 \begin{equation*}
		 R(Z,s)=
\begin{cases}
            +\infty, &\text{ if } s\geq 0\\
			\frac{b_Z}{a_Z}-1, &\text{ if } -\infty < s < 0, \\
-\infty, & \text{ if } s=-\infty.
		  \end{cases}
\end{equation*}
where $a_Z:=\sup\{r\in \mathbb{R}: r\leq Z\}$ and $b_Z:=\inf\{r \in \mathbb{R}: Z\leq r\}$.
\end{proposition}
\begin{proof}
Let $\alpha$ be the GLR. Then, from \cite{CM09}, we know that
$$
\alpha(X) =\sup\Set{m \geq 0 \, : \,  \inf\limits_{Q\in\cQ^m} E^Q[X]\geq 0},
$$
where the system of supporting kernels $\set{\cQ^m}_{m\in\bR_+}$ for $\alpha$ is given explicitly by (see \cite[Proposition~4]{CM09})
$$
\cQ^m =\Set{c(1+Y) \mid c\in\bR_+, \ 0\leq Y \leq m, \ E[c(1+Y)]=1}, \quad m\in\bR_+.
$$
Using this, it can be verified that
\begin{align*}
\mathcal{A}^{m,\circ} =  \Set{ Z \in L^\infty \Mid  c\leq Z\leq c(m+1) \ \textrm{for some} \ c\in\bR_+}.
\end{align*}
Clearly, $\inf\set{m\in\bar{\bR} : Z\in\mathcal{A}^{m,\circ}} = b_Z/a_Z-1$, and so, using Proposition~\ref{prop:ScaleInvGeneral} we conclude the proof.
\end{proof}
Analogously one can establish a robust representation for dGLR.

\subsection{Optimized Certainty Equivalent}

We sketch here a conditional version of classical version of the optimized certainty equivalent. The detailed study can be done along the lines of the study the we conducted above for dGLR.

	The \emph{optimized certainty equivalent}, see \citep{ben-tal02,ben-tal01}, is an assessment index given by
	\begin{equation}\label{eq:entropic}
		OCE_t(X)=\esssup_{m \in L^\infty(\mathscr{F}_t)}\Set{m+E_{\tilde{P}}\left[ u_t\left( X_{[t,T]}-m1_{[t,T]} \right)\Mid \mathscr{F}_t \right]} ,
	\end{equation}
	where $u_t:\mathbb{R} \to \mathbb{R}\cup \set{-\infty}$ is a concave utility function\footnote{One may assume that $u_t$ can be made $\mathscr{F}_t$-state dependent. This however only a technical step.} such that $u(0)=0$ and $1 \in \partial u(0)$.
	Following the same argumentation as in \citep{ben-tal01,antonis2012}, it follows that the robust representation is of the form
	\begin{equation*}
		R\left( Q\otimes D, m \right)=m+E_{P}\left[ \sum_{k=t+1}^T \varphi_t\left( \frac{M_k \gamma_k}{M_t \mu_k} \right) \Mid \mathscr{F}_t \right], \quad Q\otimes D \in \mathcal{M}\otimes_{t}\mathcal{D},
	\end{equation*}
	where $\varphi_t$ is the convex conjugate of $-u(-\cdot)$, $M$ is the density process of $Q$ and $Q\otimes \gamma=Q\otimes D$ by means of relation \eqref{eq:onetoone:relation02}.

	As for the dynamic, of the OCE, if $u_t(x)=(1-e^{-\gamma x})/\gamma$, for a fixed $\gamma$, then the $OCE$ is the entropy and is time consistent, see \cite{AcciaioFollmerPenner2010}.
	Otherwise, being a risk measure it is a certainty equivalent, henceforth, a recursive definition along the line of Proposition \ref{prop:bellman} yields a strong time consistent assessment index.

\subsection{Weighted $V@R$ Acceptability Indices}

	Similarly as in the previous subsection we present here just a sketch of possible conditional version of weighted $V@R$ acceptability indices .

Following \citep{ChernyMadan2009}, we define $[0,1](\mathscr{F}_t)=\set{\alpha \in L^0(\mathscr{F}_t):0\leq \alpha\leq 1}$.
	This set is clearly $\sigma$-stable.
	We consider a family of functions $\Phi_m:[0,1](\mathscr{F}_t)\to [0,1](\mathscr{F}_t)$, $m\in L^0_{+}(\mathscr{F}_t)$ being
	\begin{itemize}
		\item jointly local:
			\begin{equation*}
				1_{A}\Phi_m(\alpha)+1_{A^c}\Phi_{n}(\beta)=\Phi_{1_A m+1_{A^c}n}(1_{A}\alpha+1_{A^c}\beta),
			\end{equation*}
			for every $A \in \mathscr{F}_{t}$, $m,n \in L^0_+(\mathscr{F}_t)$ and $\alpha,\beta \in [0,1](\mathscr{F}_t)$;
		\item concave: $\alpha \mapsto \Phi_{m}(\alpha)$ is concave;
		\item increasing: $\Phi_m \leq \Phi_{n}$, for every $m\leq n \in L^0_+(\mathscr{F}_t)$;
		\item normalized: $\Phi_m(0)=0$ and $\Phi_m(1)=1$, for every $m \in L^0_+(\mathscr{F}_t)$.
	\end{itemize}
	Such a family is called a \emph{conditional family of concave distortions}.
	Note that being conditionally concave and local, it follows that $\Phi_m$ is continuous.
	We define the Weighted $V@R$ acceptability index as follows
	\begin{equation}
		AIW(X):=\esssup\Set{m \in L^0_+(\mathscr{F}_t):\int_{-\infty}^{\infty}x d \Psi_m\left( F_{(X_{[t+1,T]}\mid \mathscr{F}_t)}(x) \right)\geq 0},
		\label{eq:AIW}
	\end{equation}
	where $F_{(X_{[t+1,T]}\mid \mathscr{F}_t)}(x)=\tilde{P}[X_{[t+1,T]}\leq x \mid \mathscr{F}_t]$ is the regular conditional distribution under $\tilde{P}$ of $X_{[t+1,T]}$, the integral being taken $\omega$-wise.
	Once again, following the argumentation in \citep{ChernyMadan2009}, it follows that
	\begin{equation*}
		\mathcal{A}^{m,\circ}_1:=\Set{Q\otimes D \in \mathcal{M}\otimes_{t} \mathcal{D}:E\left[ \left( -\sum_{k=t+1}^T M_k \Delta D_{k+1}-\mu_k\beta\right)\Mid \mathscr{F}_t\right]\leq \phi_m(\beta),\text{ for all }\beta \in L^0_{+}(\mathscr{F}_t)},
	\end{equation*}
	where $M$ is the density process of $Q$, $\phi_m(\beta):=\esssup_{\alpha \in [0,1](\mathscr{F}_t)}\set{\Phi_m(\alpha)-\alpha\beta}$, $m\in L^0_{+}(\mathscr{F}_t)$ and $\beta \in L^0_{+}(\mathscr{F}_t)$ is the convex conjugate of $\Phi_m$.\footnote{Clearly, $(\phi_m)$ is a jointly local family of convex increasing functions.}
	With this formulation, one may define $AIMAX$, $AIMAXMIN$, $AIMINMAX$.

\begin{appendix}
\section{Appendix}\label{appendix}
\subsection{Standard Results on $L^0$-Convex Analysis}

Notations and settings are from the Preliminaries \ref{sec01}.
Let $\mathcal{Y}$ be a set of $L^0$-linear functionals from $\mathcal{X}$ to $L^0$.
We denote by $L^0$-$\sigma\left( \mathcal{X},\mathcal{Y} \right)$ the smallest topology for which the mappings
\begin{equation*}
	X \mapsto Z\left( X \right),\quad X \in \mathcal{X}
\end{equation*}
are $L^0$-continuous for any $Z\in\mathcal{Y}$.
\begin{proposition}\label{prop:duality01}
	Let $ \mathcal{X}$ be a locally $L^0$-convex topological $L^0$-module and let $\mathcal{Y}$ be a set of $L^0$-linear functionals from $\mathcal{X}$ to $L^0$.
	Then, $\mathcal{X}$ equipped with the $L^0$-$\sigma\left( \mathcal{X},\mathcal{Y} \right)$-topology is a locally $L^0$-convex topological $L^0$-module.
\end{proposition}
\begin{proof}
	By definition, the $L^0$-$\sigma\left( \mathcal{X},\mathcal{Y} \right)$-topology on $\mathcal{X}$ is generated by the following family of neighborhoods of $0$
	\begin{equation*}
		U_{\mathcal{A},\varepsilon}:=\Set{X \in \mathcal{X} : \sup_{Z \in \mathcal{A}} \abs{Z\left( X \right)}\leq \varepsilon},
	\end{equation*}
	where $\mathcal{A}$ is a finite subset of $\mathcal{Y}$ and $\varepsilon \in L^0_{++}$.
	Since $\abs{Z(.)}$ is an $L^0$-seminorm\footnote{An $L^0$-semi norm is a functional $p:\mathcal{E}\to L^0_{+}$ such
that $p\left( mX \right)=\abs{m}p\left( X \right)$ for any $m \in L^0$ and $X \in \mathcal{E}$ and $p\left( X+Y \right)\leq
p\left( X \right)+p\left( Y \right)$ for any $X,Y \in\mathcal{E}$.}, we apply \citep[Theorem 2.4]{FilipovicKupperVogelpoth2009}.
\end{proof}
Provided that $\mathcal{Y}$ is itself an $L^0$-module,  $\mathcal{X}$ also defines a set of $L^0$-linear functionals from
$\mathcal{Y}$ to $L^0$ and therefore $\left( \mathcal{Y},L^0\text{-}\sigma\left(\mathcal{Y},\mathcal{X} \right)\right)$ is again a locally $L^0$-convex topological $L^0$-module.
Furthermore, the $L^0$-dual space of $\left( \mathcal{X},L^0\text{-}\sigma\left( \mathcal{X},\mathcal{Y} \right) \right)$ is exactly $\mathcal{Y}$.
We finally say that $\mathcal{X}$ is $L^0$-reflexive if $\mathcal{X}^{\ast\ast}=\mathcal{X}$, in which $\mathcal{X}^\ast$ is equipped with the $L^0$-$\sigma\left( \mathcal{X}^\ast,\mathcal{X} \right)$-topology.
On $\mathcal{X}^\ast\times \mathcal{X}$ we always consider the dual pairing $\langle X^\ast,X \rangle:=X^\ast\left( X \right)$.

A local function $F:\mathcal{X}\to \bar{L}^0$ is said to be
\begin{itemize}
	\item \emph{upper semicontinuous} if the upper level sets given by $\set{X \in \mathcal{X} : F(X)\geq m}$ are closed for all $m \in \bar{L}^0$;
	\item \emph{proper} if $F<\infty$ and there exists $X \in \mathcal{X}$ such that $F(X)>-\infty$.	
\end{itemize}
The concave conjugate $F^\star:\mathcal{X}^\ast \to \bar{L}^0$ of $F$ is given by
\begin{equation*}
	F^\star\left( X^\ast \right):=\essinf_{X \in \mathcal{X}}\Set{\langle X^\ast,X\rangle -F\left( X \right)},\quad X^\ast \in \mathcal{X}^\ast.
\end{equation*}
The  hypograph $hypo \left( F \right)$ of $F$ is defined as
\begin{align}
		hypo\left( F \right)&:=\Set{(X,m) \in \mathcal{X}\times L^0: F\left( X \right)\geq m}.	 \label{}
\end{align}

From now on we consider $\mathcal{X}$ to be a $\sigma$-stable, locally $L^0$-convex topological $L^0$-module
such that the set of all neighborhoods of zero is $\sigma$-stable.
From the theory of $L^0$-modules in  \cite{FilipovicKupperVogelpoth2009} we know the following.
\begin{proposition}\label{prop:FM}
	Let $F:\mathcal{X}\to \bar{L}^0$ be a proper function, then
	\begin{enumerate}
		\item $F$ is $L^0$-concave if and only if $hypo\left( F \right)$ is $L^0$-convex and $F$ is $L^0$-local.
		\item $F^\star$ is $L^0$-concave and $L^0$-upper semicontinuous for any $F$.
		\item If $F$ is an $L^0$-proper concave upper semicontinuous function then $F^{\star\star}=F$.
			\end{enumerate}
\end{proposition}
\begin{definition}
 For a non-empty family $(A_i)_{i\in I} \in \mathscr{G}$  the essential supremum  $\esssup\{A_i:i\in I\}$ is defined to be the
element $B\in \mathscr{G}$ with
\begin{enumerate}
 \item $A_i \subseteq B$ for all $i$.
\item For all $C\in \mathscr{G}$ fulfilling $1.$ and $C\subseteq B$ holds $P[B\setminus C]=0 $.
\end{enumerate}
Further we define $\esssup\{(\emptyset)\}=\emptyset$.
\end{definition}
The next lemma was proven in \citep[Lemma 2.9]{FilipovicKupperVogelpoth2009}.
\begin{lemma}
 Every non-empty family $\mathscr{A}=(A_i)_{i\in I}$ has an essential supremum.
If for all $i,j$ also $A_i\cup A_j \in \mathscr{A}$, then there exists an increasing sequence $(A^n)$ in $\mathscr{A}$ such that
$\esssup(\mathscr{A})=\bigcup_{n \in \mathbb{N}}A^n$.
\end{lemma}

\subsection{Conditional Inverse of Increasing Functions}\label{appendix:01}
In this section, for $n,m \in \bar{L}^0$, we use the convention that $n<m$ if $P[n<m]=1$.

For a local, increasing\footnote{That is, $F(m)\geq F(m^\prime)$ whenever $m\geq m^\prime$.} function $F:\bar{L}^0\to \bar{L}^0$ we define its left- and right-continuous version as
\begin{align}
	F^-(m)&:=1_{\widetilde A_m}\esssup\Set{F(n):n \in \bar{L}^0\text{ and }n<m \text{ on }\widetilde A_m}-1_{\widetilde A_m^c}\infty \label{eq:app01},\\
	F^+(m)&:=1_{\widetilde B_m}\essinf\Set{F(n):n \in \bar{L}^0\text{ and }n>m \text{ on }\widetilde B_m}+1_{\widetilde B_m^c}\infty \label{eq:app02},
\end{align}
where $m\in \bar{L}^0$ and $\widetilde A_m=\set{m>-\infty}$,  and $\widetilde B_m=\set{m<\infty}$.
Due to locality and the definition of $F^\pm$ it holds that
\begin{equation}
	F^+(m)\leq F^-(m^\prime),\quad \text{ for } m,m^\prime \in \bar{L}^0\text{ with }m<m^\prime.
	\label{eq:fundamental}
\end{equation}
\begin{definition}\label{defn:inverse}
	For a local, increasing function $F:\bar{L}^0 \to \bar{L}^0$, a local, increasing function $G:\bar{L}^0\to \bar{L}^0$ is called a \emph{conditional inverse} of $F$ if
	\begin{equation}\label{eq:inverse}
		\begin{cases}
			F^-\left(G(s)  \right)\leq s\leq F^+\left( G\left( s \right) \right)&,\text{on } \Set{F(-\infty) <s < F(\infty)},\\
			G(s)=-\infty&,\text{on }\Set{s < F(-\infty)},\\
			G(s)=\infty&,\text{on }\Set{F(\infty)< s},
		\end{cases}
	\end{equation}
	for every $s \in \bar{L}^0$.
\end{definition}
\begin{remark}
	The definition of a conditional inverse does not postulate any condition as for the values of $G$ on the boundary of the range of $F$.
	Being increasing, it simply means that $-\infty \leq G(F(-\infty))\leq G^+(F(-\infty))$ and $G^-(F(\infty))\leq G(F(\infty))\leq +\infty$.
	We can not require $G(F(-\infty))=-\infty$ or $G(F(-\infty))=G^+(F(-\infty))$ for instance.
	This is important since by the definition of the left- and right-inverse below, the proposition \ref{prop:leftrightinverse} states that both $F^{(-1,l)}$ and $F^{(-1,r)}$ are inverses of $F$.
	However, it may well happen that $F^{(-1,l)}(F(-\infty))=-\infty< F^{(-1,r)}(F(-\infty))$ as well as $F^{(-1,l)}(F(\infty))<F^{(-1,r)}(F(\infty))=+\infty$ and a convention on the values of a conditional inverse on the boundaries of $F$ would imply that neither $F^{(-1,l)}$ and $F^{(-1,r)}$ are conditional inverse.
\end{remark}
We define the \emph{conditional left-} and \emph{right-inverse} of $F$ as
\begin{align}
	F^{(-1,l)}\left( s \right):=&1_{A_s}\essinf \set{m\in \bar{L}^0:1_{A_s}F\left( m \right)\geq 1_{A_s}s}+1_{A_s^c}\infty \label{eq:leftinverse}\\
	=&1_{A_s}\esssup \set{m\in \bar{L}^0:F\left( m \right)< s\text{ on }A_s}+1_{A_s^c}\infty \nonumber ,\\
	F^{(-1,r)}\left( s \right):=&1_{B_{s}}\esssup \set{m\in \bar{L}^0:1_{B_s}F\left( m \right)\leq 1_{B_s}s}-1_{B_s^c}\infty\label{eq:rightinverse}\\
	=&1_{B_{s}}\essinf \set{m\in \bar{L}^0:F\left( m \right)> s\text{ on }B_s}-1_{B_s^c}\infty\nonumber,
\end{align}
for $s \in \bar{L}^0$,
where\footnote{Note that $A_{s}:=\Set{\esssup_{m \in \bar{L}^0}F(m) < s}^c$ and $ B_s:=\Set{\essinf_{m \in \bar{L}^0}F(m) > s}^c$} $A_{s}:=\set{F(\infty)\geq s}$ and $B_s:=\Set{F(-\infty)\leq s}$.

\begin{lemma}\label{lem:inverselocal}
	The conditional left- and right-inverse of a local, increasing function $F:\bar{L}^0 \to \bar{L}^0$ are local, increasing functions which are left- and right-continuous, respectively.
\end{lemma}
\begin{proof}
	Consider a local, increasing function $F :\bar{L}^0\to \bar{L}^0$.
	We will prove the statement for the left-inverse $F^{(-1,l)}$, and the case of right-inverse function is done similarly.

	\begin{enumerate}[label=\textit{Step \arabic*:},fullwidth]
		\item Note that, $A_{\widetilde{s}} \supseteq A_{s}$ for every $\widetilde{s}\leq s$.
			This implies that $1_{A_s^c}\infty$ is increasing.
			Hence, a direct inspection shows that $F^{(-1,l)}$ is increasing.
		\item Next we will show that $F^{(-1,l)}$ is local.
			Pick $s,\widetilde{s}\in \bar{L}^0$ and $B \in \mathscr{G}$.
			Since $F$ is local, it follows that
			\begin{equation}
				C^c:=A^c_{1_{B}s+1_{B^c}\widetilde{s}}=\Set{F(\infty)<1_{B}s+1_{B^c}\widetilde{s}}=(B\cap A^c_{s})\cup(B^c\cap A^c_{\widetilde{s}}).
			\end{equation}
			
			Consequently, we deduce that  $C=(B\cap A_s)\cup(B^c\cap A_{\widetilde{s}})\cup(A_s\cap A_{\widetilde{s}})$.
			However, $(A_s\cap A_{\widetilde{s}})\subseteq (B\cap A_s)\cup(B^c\cap A_{\widetilde{s}})$, hence
			\begin{equation}
				C=(B\cap A_{s})\cup(B^c\cap A_{\widetilde{s}}) . \label{eq:app-C0}
			\end{equation}
			This implies that
			\begin{align}
				1_{C}(1_{B}s+1_{B^c}\widetilde{s}) & =1_{B}1_{A_s}s+1_{B^c}1_{A_{\widetilde{s}}}\widetilde{s} \label{eq:app-C1}\\
				1_{C^c}(1_{B}s+1_{B^c}\widetilde{s}) &=1_{B}1_{A_s^c}s+1_{B^c}1_{A^c_{\widetilde{s}}}\widetilde{s}. \label{eq:app-C2}
			\end{align}
			We claim that,
			\begin{equation}\label{eq:app-C3}
				1_{B}\Set{m\in \bar{L}^0: 1_{B}1_{A_s}F(m)\geq 1_{B}1_{A_s}s}=1_{B}\Set{m\in \bar{L}^0: 1_{A_s}F(m)\geq 1_{A_s}s}.
			\end{equation}
			Indeed, inclusion $\supseteq$ is straightforward.
			For the converse inclusion, let $1_{B}\widetilde{n} \in 1_{B}\Set{m\in \bar{L}^0: 1_{B}1_{A_s}F(m)\geq 1_{B}1_{A_s}s}$.
			Note that by the definition of $A_s$, the set $\set{m\in \bar{L}^0: 1_{A_s}F(m)\geq 1_{A_s}s}$ is not empty.
			Indeed, $A_s=\Set{F(\infty)\geq s}$, hence,  $1_{A_s}F(\infty)\geq 1_{A_s}s$ showing that $\infty \in \set{m\in \bar{L}^0: 1_{A_s}F(m)\geq 1_{A_s}s}$.
			Hence, pick some $\widetilde{m}\in \Set{m\in \bar{L}^0: 1_{A_s}F(m)\geq 1_{A_s}s}$.
			Locality of $F$ yields $1_{B}\widetilde{n}+1_{B^c}\widetilde{m} \in\Set{m\in \bar{L}^0: 1_{A_s}F(m)\geq 1_{A_s}s}$.
			Multiplying by $1_{B}$, we get $1_{B}\widetilde{m} \in1_{B}\Set{m\in \bar{L}^0: 1_{A_s}F(m)\geq 1_{A_s}s}$.

			Using  \eqref{eq:app-C0}-\eqref{eq:app-C3}, and locality of $F$, we deduce
			\begin{align*}
				F^{(-1,l)}\left( 1_{B} s+1_{B^c}\widetilde{s} \right)=&1_{C}\essinf \Set{m \in \bar{L}^0: 1_CF(m)\geq 1_{C}\left(1_{B}s+1_{B^c}\widetilde{s}\right)}+1_{C^c}\infty\\
				=& 1_{B}1_{A_s}\essinf \Set{m\in \bar{L}^0: 1_{B}1_{A_s}F(m)\geq 1_{B}1_{A_s}s}\\
				& \quad+1_{B^c}1_{A_{\widetilde s}}\essinf \Set{m\in \bar{L}^0: 1_{B^c}1_{A_{\widetilde{s}}}F(m)\geq 1_{B^c}1_{A_{\widetilde{s}}}\widetilde{s}}\\
				&\qquad +1_{B}1_{A_s^c}\infty+1_{B^c}1_{A_{\widetilde{s}}^c}\infty\\
				=&1_{B} \left( 1_{A_s}\essinf \Set{m\in \bar{L}^0: 1_{A_s}F(m)\geq 1_{A_s}s}+1_{A_s^c}\infty \right) \\
				&\quad+1_{B^c}\left( 1_{A_{\widetilde{s}}} \essinf \Set{m\in \bar{L}^0: 1_{A_{\widetilde{s}}}F(m)\geq 1_{A_{\widetilde{s}}}\widetilde{s}}+1_{A_{\widetilde{s}}^c}\infty \right)\\
				=&1_{B}F^{(-1,l)}(s)+1_{B}F^{(-1,l)}(\widetilde{s}).
			\end{align*}
			Hence $F^{(-1,l)}$ is local.
		\item	Finally, we will show that $F^{(-1,l)}$ is left-continuous.
			Let $s \in \bar{L}^0$.

By the definition of $F^{(-1,l)}$ and locality of $F$, clearly $F^{(-1,l)}(s)=-\infty$ on the set $C_s^c=\set{s=-\infty}$.
			Consider now $D_s=C_s\cap \set{F(\infty)\geq s}=\set{s>-\infty}\cap \set{F(\infty)\geq s}$, and $D_{\widetilde{s}}:=C_s\cap\set{F(\infty)\geq \widetilde{s}}$, for some $\widetilde{s}\in\bar{L}^0$.
			Note that $D_s \subseteq D_{\widetilde{s}}$ for any $\widetilde{s}$ such that $\widetilde{s}<s$ on $D_s$.

Denote by $\cS$ the set of those $\widetilde{s}\leq s$ such that $\widetilde{s}<s$ on $D_s$.  Note that $\cS\neq \emptyset$.
Let $\widetilde{s}\in\cS$, and suppose that $\esssup_{\widetilde{s}\in\cS}F^{(-1,l)}(\widetilde{s})<\widetilde{m}<F^{(-1,l)}(s)$ on some set $D\subseteq D_s$.
			By the definition of the left-inverse, and locality of $F$, it follows that $\widetilde{s}<F(\widetilde{m})<s$ on $D$ for every $\widetilde{s}\in\cS$, which is not possible unless $P[D]=0$.
			Hence, $\esssup_{\widetilde{s}\in\cS}F^{(-1,l)}(\widetilde{s})=F^{(-1,l)}(s)$ on $D_s$. From here, using locality of $F$, we also have that
$\esssup_{\widetilde{s}<s}F^{(-1,l)}(\widetilde{s})=F^{(-1,l)}(s)$ on $D_s$.
			Next, let us consider the set $E_s:=C_s\cap\set{F(\infty)<s}$.
			Since $F(\infty)<s$ on $E_s$, there exists $\widetilde{s}\in\bar{L}^0$ such that $\widetilde{s}<s$ on $E_s$, and $E_{\widetilde{s}}:=C_s\cap \set{F(\infty)<\widetilde{s}}=E_s$.
			Therefore, by the definition of $F^{(-1,l)}$ we conclude that $F^{(-1,l)}(\widetilde{s})=F^{(-1,l)}(s)$ for any $\widetilde{s}<s$ on $E_s=E_{\widetilde{s}}$, which  consequently shows that $F^{(-1,l)}$ is left continuous on $E_s$. \\
			Finally, since $C_s^c,D_s,E_s$ forms a partition of $\Omega$, and $F^{(-1,l)}$ is left-continuous on each of the sets from the partition, combined with  locality of $F^{(1-,l)}$, we deduce that $F^{(-1,l)}$ is left-continuous.
	\end{enumerate}
	The case of $F^{(-1,r)}$ follows analogously.
\end{proof}
\begin{remark}
	The sets $A_s,B_s$ are used to guarantee the locality of the right-and left-inverse, respectively.
	Indeed, suppose that we would define $F^{(-1,l)}(s)=\essinf \set{m\in \bar{L}^0:F\left( m \right)\geq s}$.
	Then it is possible to get a non-local inverse.
	For example, let $A \in \mathscr{G}$ with $0<P[A]<1$ and $F(m):=1_{A}2m+1_{A^c}$ which is increasing and local.
	Then, $F^{(-1,l)}(1_{A} 2)=1_{A}-1_{A^c}\infty$, whereas $F^{(-1,l)}(2)=\essinf{\emptyset}=+\infty$, and thus $1_{A}F^{(-1,l)}(1_{A}2)=1_{A}\neq1_A\infty= 1_{A}F^{(-1,l)}(2)$, which implies that $F^{(-1,l)}$ would not be local.
\end{remark}

\begin{proposition}\label{prop:leftrightinverse}
 Let  $F:\bar{L}^0 \to \bar{L}^0$ be a local, increasing function. Then, the following properties hold true:
\begin{description}
\item{(i)} Any conditional inverse $G$ of $F$ satisfies
	\begin{equation}
		F^{(-1,l)}=G^-\leq G\leq G^+=F^{(-1,r)};
		\label{eq:myinverse}
	\end{equation}
\item{(ii)} $F^{(-1,l)}$ and $F^{(-1,r)}$ are also both conditional inverse of $F$;

\item{(iii)} $F$ is a conditional inverse of any of its conditional inverses;

\item{(iv)}	For any $m,s \in \bar{L}^0$ we have that
	\begin{align}
		F^-(m)\leq s\quad &\Longleftrightarrow \quad m\leq F^{(-1,r)}\left( s \right)\label{C3}\\
		F^+(m)\geq s\quad &\Longleftrightarrow \quad m\geq F^{(-1,l)}\left( s \right)\label{C4}.
	\end{align}
\end{description}	
\end{proposition}
\begin{remark}\label{rem:inverseleftright}
	Note that since $F$ is a conditional inverse of any of its conditional inverse, \eqref{eq:myinverse} implies that
	\begin{equation}
		\begin{split}
			F^-&=\left( F^{(-1,l)} \right)^{(-1,l)}=G^{(-1,l)}=\left( F^{(-1,r)} \right)^{(-1,l)},\\
			F^{+}&=\left( F^{(-1,r)} \right)^{(-1,r)}=G^{(-1,r)}=\left( F^{(-1,l)} \right)^{(-1,r)}.
		\end{split}
		\label{eq:inverseleftright}
	\end{equation}
\end{remark}

\begin{proof}
	Consider a local, increasing function $F :\bar{L}^0\to \bar{L}^0$ and a conditional inverse $G$ of $F$.
	\begin{enumerate}[label=\textit{Step \arabic*:},fullwidth]
		\item Let us show that
			\begin{equation}
				F^{(-1,l)}\leq G^-\leq G\leq G^+\leq F^{(-1,r)}.
				\label{eq:fundamental2}
			\end{equation}
			The fact that $G^-\leq G\leq G^+$ follows from the definition of left- and right-continuous version and from the fact that $G$ is increasing.
By Lemma \ref{lem:inverselocal}, we have  that $F^{(-1,l)}$ and $F^{(-1,r)}$ are local, increasing  and left- and right-continuous, respectively.

			Let us show now that $F^{(-1,l)}\leq G^-$.
			Since $F^{(-1,l)}$ is left-continuous, and both $F^{(-1,l)}$ and $G$ are increasing, it is  sufficient to show that $F^{(-1,l)}(s)\leq G(s)$, for every $s\in \bar{L}^0$. Assume that $s\in \bar{L}^0$.			The definition of $F^{(-1,l)}$ shows that $F^{(-1,l)}(s)=-\infty\leq G(s)$ on $\set{s\leq F(-\infty)}$.
			Since $G$ is an inverse of $F$, it follows that $G(s)=\infty\geq F^{(-1,l)}(s)$ on\footnote{Here also the set to be considered is $\set{s>F(\infty)}$ and not $\set{s\geq F(\infty)}$.} $\set{s>F(\infty)}$.
			On $\set{F(-\infty)<s<F(\infty)}$, suppose that there exists $\widetilde{m}\in L^0$ such that $F^{(-1,l)}(s)>\widetilde{m}>G(s)$ on some set $A\subseteq \set{F(-\infty)<s<F(\infty)}$.
			On the one hand,  by definition of $F^{(-1,l)}$ follows that $s>F(\widetilde{m})$ on $A$.
			On the other hand, since $\widetilde{m}>G(s)$ on $A$ it follows by means of \eqref{eq:fundamental} that $F(\widetilde{m})\geq F^-(\widetilde{m})\geq F^+(G(s))$.
			Thus, $s>F^{+}(G(s))$ on $A \subseteq \set{F(-\infty)<s<F(\infty)}$, which  contradicts the fact that $G$ is an inverse of $F$.
			Hence, $A$ has to be of probability $0$, and so, we proved that $F^{(-1,l)}\leq G$ on $\set{F(-\infty)<s<F(\infty)}$.

Finally, note that,  since $F^{(-1,l)}$ is left-continuous and $F^{(-1,l)}(s')\leq G(F(\infty))$ for any $s'<F(\infty)$, we have that $F^{(-1,l)}(F(\infty))\leq G(F(\infty))$.
The latter, together with locality of $F^{(-1,l)}$ and $G$, imply that $F^{(-1,l)}(s) \leq G(s)$ on set $\set{s=F(\infty)}$.
Hence, we conclude that $F^{(-1,l)}\leq G$.

A similar argumentation shows that $G^+\leq F^{(-1,r)}$ and therefore \eqref{eq:fundamental2} holds true.

		\item Let us show that
\begin{align}
(F^{(-1,l)})^{+}=F^{(-1,r)} \label{eq:app-C5}\\
(F^{(-1,r)})^{-}=F^{(-1,l)}. \label{eq:app-C6}
\end{align}
			Since $F^{(-1,l)}\leq F^{(-1,r)}$ and the latter is right-continuous, it follows that $(F^{(-1,l)})^+\leq F^{(-1,r)}$.
			On the other hand, for any $s<\widetilde{s}$, we have that  $F^{(-1,r)}(s)\leq F^{(-1,l)}(\widetilde{s})$.
			Indeed, on $A_{\widetilde{s}}^c$, it holds $F^{(-1,l)}(\widetilde{s})=\infty\geq F^{(-1,r)}(s)$.
			On $B_s^c$, it holds $F^{(-1,r)}(\widetilde{s})=-\infty \leq F^{(-1,l)}(s)$.
			Finally, on $C=(A^c_{\widetilde{s}}\cup B_{s}^c)^c=A_{\widetilde{s}}\cap B_{s}$, it holds $F(-\infty)\leq s<\widetilde{s}\leq F(\infty)$.
			Using now $s<\widetilde{s}$, and the definition of $F^{(-1,l)}$ and $F^{(-1,r)}$, since $C\subseteq A_{\widetilde{s}}$ and $C\subseteq B_s$, it yields
			\begin{equation*}
				1_{C}\Set{m \in \bar{L}^0: F(m)>s \text{ on }C }\supseteq 1_C\Set{m \in \bar{L}^0: 1_{C}F(m)\geq 1_{C}\widetilde{s}}.
			\end{equation*}
			Taking the essential infimum on both sides shows that $ 1_{C}F^{(-1,r)}(s) \leq 1_C F^{(-1,l)}(\widetilde{s}) $ for any $s<\widetilde{s}$.
			This together with $(F^{(-1,l)})^+ \leq F^{(-1,r)}$ implies by the definition of the right-continuous version that $1_C(F^{(-1,l)})^+=1_CF^{(-1,r)}$.
			Since $P[C\cup A^c_{\widetilde{s}}\cup B^c_s]=1$, it follows that $(F^{(-1,l)})^{+}=F^{(-1,r)}$.
			
			A similar argumentation yields $(F^{(-1,r)})^-=F^{(-1,l)}$.

		\item We deduce from \eqref{eq:fundamental2}, \eqref{eq:app-C5} and \eqref{eq:app-C6} that $F^{(-1,l)}=G^-$ and $F^{(-1,r)}=G^+$.
Therefore, \eqref{eq:myinverse} follows.
			
Let us prove that $F^{(-1,l)}$ and $F^{(-1,r)}$ are both conditional inverses of $F$. Towards this end, we first observe that \eqref{eq:fundamental2} together with
			Lemma \ref{lem:inverselocal} yield that $G^-$ and $G^+$ are local, increasing functions.
			Since $G(s)=-\infty$ on $\set{s<F(-\infty)}$ and $G(s)=\infty$ on $\set{s>F(\infty)}$, it follows immediately that the same holds for the left- and right-continuous versions of $G$.
			Using the fact that $G$ is a conditional inverse, monotonicity of $F^\pm$ yields
			\begin{equation*}
				F^-(G^-(s))\leq F^-(G(s))\leq s\leq F^{+}(G(s))\leq F^+(G^+(s)),
			\end{equation*}
			on $\set{F(-\infty)<s<F(\infty)}$.

		On the other hand, since $F^-, F^+, G$ are increasing, and $G$ is a conditional inverse, we deduce that
			\begin{equation*}
				\begin{split}
					F^-(G^+(s))&= F^-(\essinf_{\widetilde{s}>s}G(\widetilde{s}))\leq \essinf_{\widetilde{s}>s}F^-(G(\widetilde{s}))\leq \essinf_{\widetilde{s}>s}\widetilde{s}=s;\\
					F^+(G^-(s))&= F^+(\esssup_{\widetilde{s}<s}G(\widetilde{s}))\geq \esssup_{\widetilde{s}<s}F^+(G(\widetilde{s}))\geq \esssup_{\widetilde{s}<s}\widetilde{s}=s,
				\end{split}
			\end{equation*}
			on $\set{F(-\infty)<s<F(\infty)}$.
			Thus $F^{(-1,l)}=G^-$ and $F^{(-1,r)}=G^+$ are both conditional inverse of $F$.

		\item Let us show that $F$ is a conditional inverse of any of its conditional inverses.
			Let $G$ be a conditional inverse of $F$ and let $s,m \in \bar{L}^0$.

			First, we  claim that
\begin{equation}\label{eq:app7}
s>F(m) \ \textrm{ implies that} \  G(s)\geq m.
\end{equation}
			Indeed, on $\set{s>F(\infty)}$, $G(s)=\infty\geq m$. Next, note that, the assumption  $s>F(m)$ implies that $\set{s<F(\infty)} = \set{F(-\infty) < s<F(\infty)}$. Hence, on $ \set{s<F(\infty)}$, we have that $F^+(G(s))\geq s>  F(m)$ which implies $G(s)\geq m$.
			Finally, on $\set{s=F(\infty)}$, it follows that $F(\infty)=s>F(m)$.
			Therefore, $G(F(\infty))\geq F^{(-1,l)}(F(\infty))=\essinf \set{n \in \bar{L}^0:F(n)  \ \geq F(\infty)}\geq m$, since $F(\infty)>F(m)$.
			Hence, \eqref{eq:app7}.

			By \eqref{eq:app7}, and the definition of the right-continuous version, it follows that $G^+(F(m))\geq m$.
			A similar argumentation shows that $G^-(F(m))\leq m$ .
			Consequently,
			\begin{equation}\label{eq:app08}
				G^-(F(m))\leq m\leq G^+(F(m))\quad\text{on }\set{G(-\infty)<m<G(\infty)}.
			\end{equation}

Clearly $G(\infty)<\infty$ on set $\set{m>G(\infty)}$.
			By the definition of the conditional inverse, it follows that $F(\infty)=\infty$, on $\set{m>G(\infty)}$.
			Hence, by the first line of \eqref{eq:inverse}  we conclude that $F^+(G(\infty))=\infty$ on $\set{m>G(\infty)}$, which consequently implies that $F(m)=\infty$ on $\set{m>G(\infty)}$. Similarly, we get that $F(m)=-\infty$ on set $\set{m<G(-\infty)}$.
			From here, and \eqref{eq:app08}, we conclude that $F$ is a conditional inverse of any of its conditional inverse.

		\item Finally, let us show that \eqref{C3} and \eqref{C4} are satisfied.

Consider $m,s \in \bar{L}^0$.
			By the definition of $F^{(-1,l)}$, we have at once  \footnote{Note that by definition of the left-continuous version, $F(-\infty)\leq F^-(m)\leq s$ on $\set{m>-\infty}$} that on the set $\set{m>-\infty}$, $F^-(m)\leq s$ implies $m\leq F^{(-1,r)}(s)$. Clearly, this implication also holds true on the set $\set{m=-\infty}$.
Similarly, we deduce that $F^+(m)\geq s$ implies $m\geq F^{(-1,l)}(s)$.

The  converse implications follow by applying the last two implications to $G$ and then using \eqref{eq:myinverse} along with Remark \ref{rem:inverseleftright}.

	\end{enumerate}
\end{proof}

\subsection{Proof of Proposition \ref{prop:inverse}}\label{appendix:02}
\begin{proof}
 Let us first observe that Proposition \ref{prop:leftrightinverse} implies that there is a one-to-one correspondence between functions $F\, :\, \bar L^0 \to \bar L^0$, that are local, increasing and right-continuous, and their conditional right-inverses. In other words, the conditional right-inverse operator is a bijection between the sets of such functions. From this we deduce that if $\pi :\mathcal{K}^\circ \times \bar{L}^0\to \bar{L}^0$ is local in the second argument and if it satisfies \ref{cond:pen01}, then, its conditional right-inverse, say $R :\mathcal{K}^\circ \times \bar{L}^0\to \bar{L}^0$ is local in the second argument and satisfies \ref{cond:rm01}; moreover, the conditional right-inverse of $R$ is equal to $\pi.$

In the rest of the proof we shall show that, additional properties of $\pi$ are satisfied if and only if corresponding additional properties of $R$ are satisfied, e.g. \ref{cond:pen00}--\ref{cond:pen01} $\Leftrightarrow $ \ref{cond:rm00}--\ref{cond:rm01}, \ref{cond:pen00}--\ref{cond:pen01}, \ref{cond:pen02} $\Leftrightarrow $ \ref{cond:rm00}-- \ref{cond:rm01}, \ref{cond:rm02}, etc.

We start with showing that $R$ is jointly local if $\pi$ is jointly local.
	Take $X^\ast \in \mathcal{X}^\ast$, $s\in \bar{L}^0$ and $A \in \mathscr{G}$.
By similar argumentations as in the proof of locality from Proposition~\ref{prop:leftrightinverse}, and by the joint locality of $\pi$, we deduce that
	\begin{align*}
		1_{A}R\left( X^\ast,s \right)&=1_{A}R\left( X^\ast,1_{A}s \right)\\
		&=1_{A}1_{B_{1_{A}s}}\esssup\Set{m \in \bar{L}^0:1_{B_{1_{A}s}}\pi\left( X^\ast,m \right)\leq 1_{B_{1_{A}s}}1_{A}s}-1_{A}1_{B_{1_{A}s}^c}\infty \\
		&=1_{A}1_{B_{1_{A}s}}\esssup\Set{m \in \bar{L}^0:1_A1_{B_{1_{A}s}}\pi\left( X^\ast,m \right)\leq 1_{B_{1_{A}s}}1_{A}s}-1_{A}1_{B_{1_{A}s}^c}\infty \\
		&=1_{A}1_{B_{1_{A}s}}\esssup\Set{m \in \bar{L}^0:1_{B_{1_{A}s}}\pi\left( 1_AX^\ast,m \right)\leq 1_{B_{1_{A}s}}1_{A}s}-1_{A}1_{B_{1_{A}s}^c}\infty \\
		&=1_{A}R\left( 1_{A}X^\ast,1_{A}s \right),
	\end{align*}
where $B_{1_As} = \set{\pi(X^\ast,m) \leq 1_As}$.
This shows the joint locality of $R$.
Assuming that $R$ is jointly local, the joint locality of $\pi$ is proved similarly.

	The equivalences between \ref{cond:pen02}--\ref{cond:pen04} and \ref{cond:rm02}--\ref{cond:rm04} are proved similarly as in \citep[Lemma C.2]{DrapeauKupper2010} after corresponding adjustments to the conditional case.
Indeed, under condition \ref{cond:pen00}, the fact that $\pi(\cdot,m)$ is upper semicontinuous and concave for every $m \in \bar{L}^0$ is equivalent to the fact that the hypograph of $\pi$
	\begin{equation*}
		\Set{(X^\ast,s)\in \mathcal{K}^\circ\times L^0:\pi(X^\ast,m)\geq s}
	\end{equation*}
	is closed and convex for every $m \in \bar{L}^0$.
	Using $\eqref{C4}$, this is equivalent to the fact that the set
	\begin{equation*}
		\Set{(X^\ast,s)\in \mathcal{K}^\circ\times L^0:m\geq R^-(X^\ast,s)}
	\end{equation*}
	is closed and convex for every $m \in \bar{L}^0$,  which implies that $R^-$ is jointly lower semicontinuous and quasiconvex.
	Furthermore $R^-$ is jointly quasiconvex if and only if $R$ is jointly quasiconvex.

	Similarly, one can show that $\pi$ is positive homogeneous if and only if $R(\lambda X^\ast,s)=R(X^\ast,s/\lambda)$ for every $\lambda\in L^0_{++}$.

Finally, we will show the equivalence between \ref{cond:pen03} and \ref{cond:rm03}, under the assumption that \ref{cond:pen00}, \ref{cond:pen01}, and respectively \ref{cond:rm00}, \ref{cond:rm01} are satisfied. Note that condition \ref{cond:pen03} is equivalent to the following condition
	\begin{align*}
			\Big( \pi\left( X^\ast,m \right)=\infty\quad \text{for some }m\in \bar{L}^0, \ X^\ast \in \mathcal{K}^\circ \Big) \ \Longrightarrow \
			\Big( \pi\left( Y^\ast,m \right)=\infty\quad \text{for all }Y^\ast \in \mathcal{K}^\circ\Big),
	\end{align*}
which, consequently, is equivalent to
	\begin{align*}
		\Big( \pi\left( X^\ast,m \right)\geq s\quad \text{for all }s \in L^0\text{, and for some }m\in \bar{L}^0, \ X^\ast \in \mathcal{K}^\circ\Big) \\
 \Longrightarrow \Big( \pi\left( Y^\ast,m \right)\geq s \quad\text{ for all }s \in L^0,\text{ and for all }Y^\ast \in \mathcal{K}^\circ\Big).
	\end{align*}
By \eqref{C4}, it follows that the latter implication is equivalent to
	\begin{align*}
		\Big( m\geq R^-\left( X^\ast,s \right)\quad \text{for all }s \in L^0\text{, for some }m\in \bar{L}^0, \ X^\ast \in \mathcal{K}^\circ \Big) \\
	\Longrightarrow \Big(	m\geq R^-\left( Y^\ast,s \right) \quad\text{ for all }s \in L^0, \text{ and for all }Y^\ast \in \mathcal{K}^\circ \Big).
	\end{align*}
	Noticing that $R^-(X^\ast,\infty)=\esssup_{s \in L^0} R(X^\ast,s)$, we deduce that the last condition is equivalent to
	\begin{align}
		\Big( m\geq R^-\left( X^\ast,\infty \right)=\esssup_{s \in L^0}R(X^\ast,s)\quad \text{for some }m\in \bar{L}^0, \ X^\ast \in \mathcal{K}^\circ \Big)\nonumber\\
		\Longrightarrow \Big(
		m\geq R^-\left( Y^\ast,\infty \right)=\esssup_{s \in L^0}R(Y^\ast,s), \ \text{ and for all }Y^\ast \in \mathcal{K}^\circ \Big).  \label{eq:piinverse2}
	\end{align}
Taking in the last implication $m=R^-(X^\ast,\infty)$, we get that 	$R^-(X^\ast,\infty)\geq R^-(Y^\ast,\infty)$ for any $Y^\ast$. Applying the equivalence consequently to $m=R^-(Y^\ast,\infty)$, we conclude that
\begin{equation}\label{eq:piinverse3}
		R^-\left( X^\ast,\infty \right)=R^-(Y^\ast,\infty)\quad \text{for all }X^\ast,Y^\ast \in \mathcal{K}^\circ.
\end{equation}
Clearly, if  \eqref{eq:piinverse3} holds true, then implication \eqref{eq:piinverse2} also holds true, and hence \eqref{eq:piinverse2} is equivalent to \eqref{eq:piinverse2}.
Thus, $\pi$ satisfies \ref{cond:pen03} if and only if $R$ satisfies \ref{cond:rm03} which completes the proof.

\end{proof}

\subsection{Proof of Proposition~\ref{thm:01}}\label{appendix:03}

Before proving the Proposition~\ref{thm:01}, we first give the definition of the conditional characteristic function, followed by the Proposition~\ref{prop:support} that contains some relevant properties of the conditional characteristic function.

\begin{definition}\label{def:support}
	Let $\mathcal{C}$ be a $\sigma$-stable subset of $\mathcal{X}$.
	For $X \in \mathcal{X}$ we define $A(X)=\esssup\Set{B \in \mathscr{G}: 1_{B}X\in 1_{B}\mathcal{C}}$.
	The function $\chi_{\mathcal{C}}:\mathcal{X}\to \bar{L}^0$ given by
	\begin{equation}
		\chi_{\mathcal{C}}\left( X \right)=-1_{A^c(X)}\infty=
		\begin{cases}
			0&\text{on } A(X)\\
			-\infty&\text{on }A^c(X)
		\end{cases}
		,\quad X \in \mathcal{X},
		\label{def:supportfion}
	\end{equation}
	is called the \emph{conditional characteristic function} of $\mathcal{C}$.
\end{definition}
Note that the conditional characteristic function is a mapping from  $\mathcal{X}$ to $\bar{L}^0$.
\begin{proposition}\label{prop:support}
	Let $\mathcal{C}$ be a $\sigma$-stable set.
	Then, $\chi_{\mathcal{C}}$ is a local function.
	Furthermore,
	\begin{itemize}
		\item $\mathcal{C}$ is nonempty if and only if $\chi_{\mathcal{C}}$ is proper;
		\item $\mathcal{C}$ is monotone if and only if $\chi_{\mathcal{C}}$ is monotone;
		\item $\mathcal{C}$ is convex if and only if $\chi_{\mathcal{C}}$ is concave;
		\item $\mathcal{C}$ is a cone if and only if $\chi_{\mathcal{C}}$ is positive homogeneous;
		\item $\mathcal{C}$ is closed if and only if $\chi_{\mathcal{C}}$ is upper semicontinuous.
	\end{itemize}	
\end{proposition}
\begin{proof}
	For $B \in \mathscr{G}$ and $X \in \mathcal{C}$, since $\mathcal{C}$ is $\sigma$-stable, it holds
	\begin{multline*}
		B\cap A\left( X \right)=\esssup\set{\widetilde{B}\cap B:\widetilde{B}\in \mathscr{G}\text{ and }1_{\widetilde{B}}X\in 1_{\widetilde{B}}\mathcal{C}}\\
		=B\cap \esssup\Set{\widetilde{B} \in \mathscr{G}:1_{\widetilde{B}}1_{B}X \in 1_{\widetilde{B}}1_{B}\mathcal{C}}=B\cap A\left( 1_{B}X \right).
	\end{multline*}
	This implies that $B\cap A^c(X) = B\cap A^c(1_BX)$, and hence, $1_{B}\chi_{\mathcal{C}}\left( 1_{B}X \right)=1_{B}\chi_{\mathcal{C}}\left( X \right)$, and therefore $\chi_{\mathcal{C}}$ is local.
	By definition, $\chi_{\mathcal{C}}<+\infty$.
	On the other hand, $A(X)$ is of measure zero for every $X$ if and only if $\mathcal{C}$ is the empty set,
therefore $\chi_{\mathcal{C}}$ is proper if and only if $\mathcal{C}$ is nonempty.
	The monotonicity of $\cC$ implies the monotonicity of $\chi_\cC$ is immediate by the definition of $A(X)$, $X \in \mathcal{X}$.
	Since $X \in \mathcal{C}$ if and only if $\chi_{\mathcal{C}}(X)=0$, the converse implication also follows.
	Using $\sigma$-stability of $\mathcal{C}$, it can be showed that $A( \lambda X+(1-\lambda)Y )\supseteq (A( X )\cap A (Y ))$ if and only if $\mathcal{C}$ is convex, and so $\chi_{\mathcal{C}}$ is a concave function if and only if $\mathcal{C}$ is convex.
	Similarly one proves that $\mathcal{C}$ is a cone if and only if $\chi_{\mathcal{C}}$ is positive homogeneous.

	Finally, if $\mathcal{C}=\emptyset$, clearly $\chi_{\mathcal{C}}$ is  upper semicontinuous.
	Otherwise, note that
	\begin{equation*}
		\Set{X \in \mathcal{X}:\chi_{\mathcal{C}}\left( X \right)\geq m}=
		\begin{cases}
			1_{\set{m>-\infty}}\mathcal{C}+1_{\set{m=-\infty}}\mathcal{X} &\text{if }m\leq 0,\\
			\emptyset &\text{otherwise,}
		\end{cases}
	\end{equation*}
	which is a closed set for every $m\in \bar{L}^0$ if and only if $\mathcal{C}$ is closed.
\end{proof}


\begin{proof}[Proof of Proposition~\ref{thm:01}]
If $\mathcal{C}=\emptyset$, then $\pi\equiv \infty$ fulfills all required conditions.
Hence, we will consider the case $\mathcal{C}\neq\emptyset$.

\begin{enumerate}[label=\textit{Step \arabic*:},fullwidth]

	\item We first assume that $\mathcal{K}=\set{0}$, so that $\mathcal{K}^\circ=\mathcal{X}^\ast$.
		We start with the existence of $\pi$.
		Since, $\mathcal{C}\neq\emptyset$, by Proposition~\ref{prop:support}, the conditional characteristic function $\chi_{\mathcal{C}}$ is a local, proper, concave and upper semicontinuous function. Using the definition of $\chi_\cC$, we deduce that its concave conjugate $\chi^\star_{\mathcal{C}}(X^\ast):=\essinf_{X \in \mathcal{X}}\Set{\langle X^\ast,X\rangle-\chi_{\mathcal{C}}(X)}$ can be also represented as follows
		\begin{equation}\label{eq:myblabla02}
			\chi^\star_{\mathcal{C}}(X^\ast)=\essinf_{X \in \mathcal{C}}\langle X^\ast,X\rangle,\quad X^\ast \in \mathcal{X}^\ast.
		\end{equation}
		Indeed, since $X \in \mathcal{C}$ if and only if $\chi_{\mathcal{C}}(X)=0$, it clearly follows that $\chi^\star_{\mathcal{C}}(X^\ast)\leq\essinf_{X \in \mathcal{C}}\langle X^\ast,X\rangle$.
		Suppose now that there exists $X_0 \in \mathcal{X}$ such that
		\begin{equation}\label{eq:myblabla}
			1_{A}\langle X^\ast,X_0\rangle-1_{A}\chi_{\mathcal{C}}(X_0)<1_{A}\essinf_{\mathcal{C}}\langle X^\ast,X\rangle
		\end{equation}
		on some set $A$.
		Note that by locality, the definition of $\chi_{\mathcal{C}}$ and the fact that $\mathcal{C}\neq \emptyset$, we have that  $1_{A}\chi^\star_{\mathcal{C}}(X^\ast)=1_{A}\chi_{\mathcal{C}}^\star(1_{A}X^\ast)$, $1_{A}\essinf_{\mathcal{C}}\langle X^\ast,X\rangle=\essinf_{X \in 1_A\mathcal{C}}\langle 1_AX^\ast,X\rangle$ and $1_{A}\chi_{\mathcal{C}}(X)=\chi_{1_{A}\mathcal{C}}(1_{A}X)$.
		However, the strict inequality in \eqref{eq:myblabla} implies that $1_A\chi_{\mathcal{C}}(X_0)=\chi_{1_{A}\mathcal{C}}(1_AX_0)>-\infty$, that is $\chi_{1_{A}\mathcal{C}}(1_{A}X_0)=0$.
		Hence
		\begin{equation*}
			1_{A}\langle X^\ast,X_0\rangle -1_{A}\chi_{\mathcal{C}}(X_0)=1_{A}\langle X^\ast,1_{A}X_0\rangle-\chi_{1_{A}\mathcal{C}}(1_AX_0)\geq 1_{A}\essinf_{X \in 1_{A}\mathcal{C}}\langle X^\ast ,X\rangle=1_{A}\essinf_{X \in \mathcal{C}}\langle X^\ast ,X\rangle
		\end{equation*}
		showing together with \eqref{eq:myblabla} that $A$ is a set of null measure.
	
	Note that, by Proposition~\ref{prop:FM}, $\chi^\star_{\mathcal{C}}$ is upper semicontinuous and concave and clearly positive homogeneous.
	Furthermore, since $\mathcal{C}\neq \emptyset$, in view of \eqref{eq:myblabla}, it follows that $\chi_{\mathcal{C}}^\star<\infty$ and therefore maximal invariant.

	By the conditional Fenchel-Moreau Theorem (cf. \citep[Theorem 3.8]{FilipovicKupperVogelpoth2009} or Proposition~\ref{prop:FM}(3)), we have that
	\begin{equation*}
		\chi_{\mathcal{C}}\left( X \right)=\chi_{\mathcal{C}}^{\star\star}\left( X \right):=\essinf_{X^\ast \in \mathcal{X}^\ast}\set{\langle X^\ast,X\rangle -\chi_{\mathcal{C}}^\star\left( X^\ast \right)}.
	\end{equation*}
	Hence, by the definition of $\chi_{\mathcal{C}}$ and \eqref{eq:myblabla02} it follows that
	\begin{equation}
		\begin{split}
			X \in \mathcal{C}\quad \Longleftrightarrow& \quad 0\leq \chi_{\mathcal{C}}(X)=\essinf_{X^\ast \in \mathcal{X}^\ast}\set{\langle X^\ast,X\rangle -\chi_{\mathcal{C}}^\star\left( X^\ast \right)}\\
			\Longleftrightarrow&\quad \langle X^\ast,X\rangle \geq \chi^\star_{\mathcal{C}}(X^\ast)=\essinf_{Y\in \mathcal{C}}\langle X^\ast,Y\rangle ,\quad\text{for all }X^\ast \in \mathcal{X}^\ast.
		\end{split}
		\label{eq:blibli1}
	\end{equation}
	Thus, the function
	\begin{equation*}
		\pi\left( X^\ast \right):=\essinf_{X \in \mathcal{C}}\langle X^\ast,X\rangle, \quad X^\ast \in \mathcal{X}^\ast,
	\end{equation*}
	fulfills relation \eqref{robust1} and the conditions \ref{cond:pen02bis} to \ref{cond:pen04bis}.
	
	\item
		As for the uniqueness of $\pi$, let $\pi^1,\pi^2:\mathcal{X}^\ast \to \bar{L}^0$ fulfill the  conditions \ref{cond:pen02bis} to \ref{cond:pen04bis} and relation \eqref{robust1}. We will still assume that $\mathcal{K}=\set{0}$.
		If $\pi^1\left( X^\ast \right)= \infty$ for some $X^\ast \in \mathcal{X}^\ast$, then by relation \eqref{robust1}, it follows that $\mathcal{C}=\emptyset$ which implies $\pi^2(Y^\ast)=\infty$ for some $Y^\ast \in \mathcal{X}^\ast$.
		Since both are maximal invariant, it follows that $\pi^1=\pi^2=\infty$.
		Now suppose $\pi^i<\infty$.\footnote{Note that since $\pi^i$ is maximal invariant then we only need to consider two cases: $\pi^i=\infty$ and $\pi^i<\infty $, $i=1,2.$}

		We claim that, for $i=1,2,$
		\begin{equation}\label{eq:robrep4}
			\Big(1_B\pi^i(X^*)=-1_B\infty \textrm{ for all } X^*\in\cX^* \Big) \quad \Longleftrightarrow \quad 1_B\cC = 1_B\cX.
		\end{equation}
		Indeed, if $1_B \pi^i(X^\ast)=-1_B \infty$, for all $X^\ast\in\cX^\ast$, then
		\begin{equation}\label{eq:robrep5}
			1_B\langle X,X^\ast \rangle\geq 1_B\pi^i(X^\ast),
		\end{equation}
		for all $X \in \mathcal{X}$ and all $X^\ast\in\cX^\ast.$
		Since $\cC\neq \emptyset$, we take any $Y\in\cC$. By \eqref{robust1}, we get that
		$1_{B^c}\langle Y,X^\ast \rangle\geq 1_{B^c}\pi^i(X^\ast)$, for all $X^\ast\in\cX^\ast$, which combined with \eqref{eq:robrep5}, and by locality, gives us
		$$
		\langle 1_BX + 1_{B^c}Y,X^\ast \rangle\geq \pi^i(X^\ast), \ \textrm{ for all } X^\ast\in\cX^\ast,
		$$
		and thus $Z=1_BX+1_{B^c}Y\in\cC$. Moreover, $1_BX=1_BZ$, and hence since $X$ was arbitrary in $\cX$, we conclude that $1_B\cX\subseteq 1_B\cC$. The inclusion $1_B\cX\supseteq1_B\cC$ is obvious, and thus $1_B\cX=1_B\cC$.

		Assume that $1_B\cX=1_B\cC$ and that there exist $X_0^\ast\in\cX^\ast, \ B_1\subseteq B$ such that $1_{B_1}\pi^i(X_0^\ast) > - 1_{B_1}\infty$ on $B_1$.
		Take $X_0\in\cX$ such that $\langle X_0 ,X_0^\ast \rangle <0$ on $B_1$. Then, for a sufficiently large $\lambda_0\in L^0_{++}$, we have that
		$\langle \lambda_0X_0 ,X_0^\ast \rangle <\pi^i(X_0^\ast)$ on $B_1$, and hence $1_{B_1}\lambda_0X_0\notin 1_{B_1}\cC$. However, $1_{B_1}\lambda_0X_0\in1_{B_1}\cX=1_{B_1}\cC$, which yields a contradiction. Thus, the equivalence  \eqref{eq:robrep4} is established.

Next, define the sets
\begin{equation*}
A^i:=\esssup\Set{B\in \mathscr{F}:1_B\pi^i(X^\ast)=-1_B\infty\text{ for all }X^\ast \in \mathcal{X}^\ast}, \ i=1,2.
\end{equation*}
By \eqref{eq:robrep4}, we get that $A_1=A_2$.
Note that on the set $A^1=A^2$, the functions $\pi^1$ and $\pi^2$ coincides and are both equal to $-\infty$.
		Define $\widetilde{\pi}^i=1_{A}\pi^i$, where $A:=(A^1)^c=(A^2)^c$.
		These functions are concave, upper semicontinuous and local since both the $\pi^i$ are so, and proper by the definition of $A$.
		Due to the conditional Fenchel-Moreau Theorem we obtain
		\begin{equation}
			\widetilde{\pi}^i\left( X^\ast \right)=\essinf_{X \in \mathcal{X}}\Set{\langle X^\ast,X\rangle-\widetilde{\pi}^{i,\star}\left( X \right) },\quad X^\ast \in \mathcal{X}^\ast,
			\label{eq:fenchel01}
		\end{equation}
		where
		\begin{equation}
			\widetilde{\pi}^{i,\star}\left( X \right)=\essinf_{X^\ast \in \mathcal{X}^\ast}\Set{\langle X^\ast,X\rangle -\widetilde{\pi}^i(X^\ast)},\quad X \in \mathcal{X}.
			\label{}
		\end{equation}
		Since $\widetilde{\pi}^i$ is positively homogeneous, and $\widetilde{\pi}^{i,\star}$ is proper, we have that $\widetilde{\pi}^{i,\star}$ can only take the values $0$ or $-\infty$.
		Therefore,
		\begin{equation*}
			\widetilde{\pi}^{\star,i}(X)=0 \quad \Longleftrightarrow\quad  \langle X^\ast,X\rangle \geq \pi^i(X^\star)\, \text{for all } X^\ast \in \mathcal{X}^\ast
			\quad \Longleftrightarrow \quad X\in 1_{A}\mathcal{C}.
		\end{equation*}
		Hence $\widetilde{\pi}^{\star,1}=\widetilde{\pi}^{2,\star}$, which together with equation \eqref{eq:fenchel01} implies that $\widetilde{\pi}^i=\widetilde{\pi}^2$.
		Thus, $\pi^1=\pi^2$.

		\item Finally, let us consider the case where $\mathcal{K}\neq \set{0}$, that is $\mathcal{K}^\circ \neq \mathcal{X}^\ast$.
						As we already showed in Step~1, the function $\pi:\cX^*\to\bar{L}^0$ given by
			\begin{equation}
				\pi(X^\ast)=\essinf_{X \in \mathcal{C}}\langle X^\ast,X\rangle, \quad X^\ast \in \mathcal{X}^\ast,
			\end{equation}
			satisfies conditions \ref{cond:pen02bis}-\ref{cond:pen04bis}, hence its restriction on $\cK^\circ$ satisfies \ref{cond:pen02bis}-\ref{cond:pen04bis}.
			Taking into account the uniqueness proved in Step~2, the proof will be complete if we show that $\pi:\cK^\circ\to\bar{L}^0$ fulfills \eqref{robust1}.

First, we will show that for any $X^\ast \in \mathcal{X}^\ast$, we have that
			\begin{equation*}
				\pi(X^\ast)=-\infty \quad \text{on }A^c_{X^*},
			\end{equation*}
			where $A_{X^*}=\esssup\set{B \in \mathscr{G}: 1_B X^\ast \in \mathcal{K}^\circ}$.
			Indeed, by definition of the polar cone and $A_{X^*}$, it follows that there exists $Y \in \mathcal{K}$ such that
			\begin{equation*}
				\langle X^\ast ,Y\rangle <0, \quad \text{ on }A^c_{X^*}.
			\end{equation*}
			Take $\hat{X} \in \mathcal{C}$;  by monotonicity of $\cC$, we get that $\hat{X}+\lambda Y \in \mathcal{C}$ for every $\lambda>0$.
			Hence,
			\begin{equation*}
				\pi(X^\ast)=\essinf_{X \in \mathcal{C}}\langle X^\ast,X\rangle \leq \langle X^\ast,\hat{X}\rangle +\lambda \langle X^\ast,Y\rangle,\quad \text{ for every }\lambda>0.
			\end{equation*}
			Hence, letting $\lambda$ going to $\infty$, and taking into account that $\langle X^\ast,Y\rangle<0$ on $A^c_{X^*}$, we conclude that $\pi(X^\ast)$ is equal to $-\infty$ on $A^c_{X^*}$.

			Next, define $\widetilde{X}^\ast:=1_{A}X^\ast$, and note that by the definition of $A_{X^*}$, we have that $\widetilde{X}^\ast\in \mathcal{K}^\circ$.
			Since $\pi(X^\ast)=-\infty$ on $A^c_{X^*}$, by locality and the fact that $\pi(0)=0$, it follows that
			\begin{equation}\label{eq:setCproof1}
				\langle X^\ast,X\rangle \geq \pi(X^\ast)\quad \Longleftrightarrow \quad \langle \widetilde{X}^\ast,X\rangle=1_{A_{X^*}}\langle X^\ast,X\rangle \geq 1_{A_{X^*}}\pi(X^\ast)=\pi(\widetilde{X}^\ast).
			\end{equation}
Note that, by locality and the definition of $A_{X^*}$ and $\widetilde{X}^*$, we have that
$$
\cK^\circ=\set{Y^* \in\cX^* : \textrm{ there exists } X^*\in\cX^* \textrm{ such that } Y^*=1_{A_{X^*}}X^*}.
$$
Using this and \eqref{eq:setCproof1}, we conclude that
		\begin{equation*}
			X\in \mathcal{C} \quad \Longleftrightarrow \quad \langle X^\ast,X\rangle \geq \pi(X^\ast)\quad\text{for all }X^\ast \in \mathcal{K}^\circ.
		\end{equation*}

\end{enumerate}
This completes the proof.
\end{proof}

\end{appendix}

\section*{Acknowledgments}
Tomasz R. Bielecki and Igor Cialenco acknowledge support from the NSF grant DMS-0908099, and DMS-1211256.
The work of Samuel Drapeau was supported in part by MATHEON, project E.11.
Martin Karliczek acknowledge support from Konsul Karl und Dr. Gabriele Sandmann Stiftung grant.
The authors would like to thank Prof. Michael Kupper for stimulating discussions and helpful remarks.


\end{document}